\documentclass[a4paper]{gtart}
 
 \usepackage[inner=25mm, outer=25mm, textheight=225mm]{geometry}

\usepackage{graphicx, amsmath, amssymb}

\usepackage{listings}
\usepackage{color}
\definecolor{mygray}{rgb}{0.9,0.9,0.9}
\usepackage[small,bf]{caption}
\def\bkR{{\rm I\kern-.17em R}}
\def\R{\bkR}

\def\Z{\mathbb{Z}}

\def\Teich{\mathcal{T}}

\DeclareMathOperator{\SL}{SL}
\DeclareMathOperator{\GL}{GL}

\DeclareMathOperator{\PGL}{PGL}

\DeclareMathOperator{\hol}{hol}
\DeclareMathOperator{\dev}{dev}

\newcommand{\RP}{\R P^2}

	\theoremstyle{plain}

	\newtheorem{theorem}{Theorem}
	\newtheorem{lemma}[theorem]{Lemma}
	\newtheorem{proposition}[theorem]{Proposition}

\begin{document}

\title{Tessellating the moduli space of strictly convex\\ projective structures on the once-punctured torus}

\author{Robert C. Haraway, III and Stephan Tillmann}

\begin{abstract}
We show that associating the Euclidean cell decomposition due to Cooper and Long to each point of the moduli space of framed strictly convex real projective structures of finite volume on the once-punctured torus gives this moduli space a natural cell decomposition. The proof makes use of coordinates due to Fock and Goncharov, the action of the mapping class group as well as algorithmic real algebraic geometry. We also show that the decorated moduli space of framed strictly convex real projective structures of finite volume on the thrice-punctured sphere has a natural cell decomposition.
 \end{abstract}



\maketitle


\section{Introduction}

Classical Teichm\"uller space can be viewed as 
the moduli space of marked hyperbolic structures 
of finite volume on a surface. In the case of a 
punctured surface, many geometrically meaningful 
\emph{ideal cell decompositions} for its 
Teichm\"uller space are known. For instance, 
quadratic differentials are used for the 
construction attributed to Harer, Mumford and 
Thurston~\cite{Harer1986}; hyperbolic geometry 
and geodesic laminations are used by Bowditch 
and Epstein~\cite{BE1988}; and 
Penner~\cite{Penner1987} uses Euclidean cell 
decompositions associated to the points in 
\emph{decorated} Teichm\"uller space. The 
decoration arises from associating a positive 
real number to each cusp of the surface. All of 
these decompositions are natural in the sense 
that they are invariant under the action of the 
mapping class group (and hence descend to a 
cell decomposition of the moduli space of 
unmarked structures) and that they do not 
involve any arbitrary choices.

A hyperbolic structure is an example of a 
\emph{strictly convex projective structure}, 
and two hyperbolic structures are equivalent as 
hyperbolic structures if and only if they are 
equivalent as projective structures. Let 
$S_{g,n}$ denote the surface  of genus $g$ with 
$n$ punctures. We will always assume that $2g+n>2,$ 
so that the surface has negative Euler 
characteristic. Whereas the classical Teichm\"uller 
space $\Teich(S_{g,n})$ is homeomorphic with 
$\R^{6g-6+2n},$ Marquis~\cite{Marq1} has shown that 
the analogous moduli space $\Teich_+(S_{g,n})$ 
of marked strictly convex projective structures of finite 
volume on $S_{g,n}$ is homeomorphic with 
$\R^{16g-16+6n}.$ 

Recently, Cooper and Long~\cite{CL} generalised the 
key construction of Epstein and Penner~\cite{MR918457}, 
which was used in Penner's decomposition of decorated 
Teichm\"uller space~\cite{Penner1987}. Cooper and 
Long~\cite{CL} state that their construction can be 
used to define a decomposition of the decorated moduli 
space $\widetilde{\Teich}_+(S_{g,n}),$ but that it is 
not known whether all components of this decomposition 
are cells.

As in the classical setting, there is a principal 
$\R^n_+$ foliated fibration 
$\widetilde{\Teich}_+(S_{g,n}) \to {\Teich}_+(S_{g,n}),$ 
and different points in a fibre above a point of 
${\Teich}_+(S_{g,n})$ may lie in different components 
of the decomposition of $\widetilde{\Teich}_+(S_{g,n}).$ 
However, if there is only one cusp, then all points in a 
fibre lie in the same component, and one obtains a 
decomposition of ${\Teich}_+(S_{g,1}).$

The proofs of cellularity in 
\cite{MR918457, Penner1987} make essential 
use of the hyperbolic metric and in particular the 
Minkowski model for hyperbolic space. One obstacle 
in finding analogous proofs that work in the setting of projective 
geometry lies in the fact that the model geometry varies. 
Whereas every hyperbolic surface is a quotient of the 
interior of the unit disc, one can only
guarantee that a strictly convex 
projective surface is the 
quotient of some open strictly 
convex domain in projective space. But as one varies the 
projective structure, the domain may change to a 
projectively inequivalent domain. Moreover, the geometry 
arises from the Hilbert metric on the domain, which in 
general is a non-Riemannian Finsler metric.

The main contribution of this paper is to give the first evidence towards a positive answer to 
the question of whether Penner's result generalises to $\widetilde{\Teich}_+(S_{g,n})$. We also introduce the concept of \emph{trigonal matrices} in \S\ref{sec:projectivity}, which allow computation of holonomy without the introduction of cube roots, and the concept of \emph{cloverleaf position} in \S\ref{sec:clovers}, which normalises domains to vary compactly.

We show that for the once-punctured torus $S_{1,1}$, the 
decomposition of $\Teich_+(S_{1,1})$ is indeed an ideal 
cell decomposition, which is invariant under the action of the mapping class group. Moreover, there is a natural bijection between the cells and the ideal cell decompositions of $S_{1,1}.$ This is stated formally as Theorem~\ref{thm:main} in \S\ref{sec:Convex hull constructions}. The analogous statement for the decorated moduli space $\widetilde{\Teich}_+(S_{0,3})$ of the thrice-punctured sphere $S_{0,3}$ is also shown (see Theorem~\ref{thm:main2} in \S\ref{sec:S03}).

In addition to giving evidence towards a generalisation of Penner's result, our methods show 
that on the one hand, the parametrisation due to Fock and 
Goncharov~\cite{FG2006, FG2007} makes the computation of the 
decomposition of moduli space feasible, and that it may also 
provide the right theoretical framework for a general proof. 
We also show that our computational tools allow a systematic 
study of deformations and degenerations of strictly convex 
projective structures  in \S\ref{sec:clovers} and discuss further directions in \S\ref{sec:conclusion}.


\section{Ideal cell decompositions of surfaces}

An \emph{ideal cell decomposition} of $S_{g,n}$ consists of 
a union $\Delta$ of pairwise disjoint arcs connecting (not 
necessarily distinct) punctures with the properties that no 
two arcs are homotopic (keeping their endpoints at the 
punctures) and that each component of $S_{g,n} \setminus \Delta$ 
is an open disc. The arcs are called \emph{ideal edges}. We 
regard two ideal cell decompositions as the same if they 
are isotopic (keeping the endpoints of all arcs at the 
punctures). The set of (isotopy classes of) ideal cell 
decompositions of $S_{g,n}$ has the structure of a 
partially ordered set, with the partial order given by 
inclusion. Given ideal cell decompositions $\Delta_1$ 
and $\Delta_2$ we always understand statements such as 
``$\Delta_1=\Delta_2,$'' ``$\Delta_1\subseteq \Delta_2$'' 
or ``$\Delta_1\cap\Delta_2\neq \emptyset$'' up to isotopy.

For instance, in the case of $S_{1,1},$ an ideal cell 
decomposition either has two ideal edges and its 
complement is an \emph{ideal quadrilateral} or it has 
three ideal edges and its complement consists of two 
\emph{ideal triangles}. We call the latter an \emph{ideal triangulation} and the former an \emph{ideal quadrilation} of $S_{1,1}.$
An ideal quadrilateral can be 
divided into two triangles in two different ways, 
depending on which diagonal is used to subdivide it. The
space of all ideal cell decompositions of $S_{1,1}$ is 
naturally identified with the infinite trivalent tree. 
Vertices of the tree correspond to ideal triangulations 
and there is an edge between two such triangulations 
$\Delta_0$ and $\Delta_1$ if and only if $\Delta_0$ is 
obtained from $\Delta_1$ by deleting an ideal edge $e$ 
(hence creating an ideal quadrilateral) and then 
inserting the other diagonal of the quadrilateral. This is called an \emph{edge flip} or \emph{elementary move}. The 
ideal cell decomposition $\Delta_1 \setminus \{e\}$ is 
associated with the edge in the tree with endpoints $\Delta_0$ and $\Delta_1.$

Floyd and Hatcher~\cite{FH1982} identify this tree with the dual tree to the 
\emph{modular tessellation} or \emph{Farey tessellation} of the hyperbolic plane. An excellent illustration of this (including the information about edge flips) can be found in Lackenby~\cite{Lackenby}. The tiles in the modular tessellation are
ideal triangles with the properties 
\begin{enumerate}
\item each vertex is a rational number or $\infty = \frac{1}{0}$, 
\item if $\frac{p}{q}$ and $\frac{r}{s}$ are two vertices of the same ideal triangle, then $ps-rq=\pm 1$,
\item the set of vertices of each ideal triangle is of the form $\{\frac{p}{q}, \frac{r}{s}, \frac{p+r}{q+s}\}$
\end{enumerate}
The full tessellation can thereby be generated from the ideal triangle with vertices $\frac{0}{1}, \frac{1}{0}, \frac{1}{1}$ and the ideal triangle with vertices $\frac{1}{0}, \frac{-1}{1}, \frac{0}{1}$. 
Moreover, the element of the mapping class group taking one ideal triangulation to another can be determined from this information.

\section{Convex hull constructions}
\label{sec:Convex hull constructions}

We summarise some key definitions and results that can 
be found in \cite{Marq1, CLT1}.
A \emph{strictly convex projective surface} is 
$S = \Omega / \Gamma,$ where $\Omega$ is an open 
strictly convex domain in the real projective plane 
with the property that the closure of $\Omega$ is contained 
in an affine patch, and $\Gamma$ is a torsion-free discrete 
group of projective transformations leaving $\Omega$ 
invariant. Since there is an analytic isomorphism 
$\PGL(3, \R)\cong \SL(3,\R),$ we may assume $\Gamma < \SL(3,\R).$

The Hilbert metric on $\Omega$ can be used to define 
a notion of volume on $S,$ and we are interested in the 
case where $S$ is non-compact but of finite volume. Then 
the ends of $S$ are cusps, and the holonomy of each cusp 
is conjugate to the \emph{standard parabolic}
\[
\begin{pmatrix} 1 & 1 & 0 \\ 0 & 1 & 1\\ 0 & 0 & 1\end{pmatrix},
\]
and its unique fixed point on $\partial \Omega$ is 
called a \emph{parabolic fixed point}.

Cooper and Long~\cite{CL} associate ideal cell decompositions 
to cusped strictly convex projective surfaces of finite volume 
as follows. Suppose $S = \Omega / \Gamma$ is homeomorphic with
 $S_{g,n}.$ The $(\SL(3, \R), \RP)$--structure of $S$ lifts to 
a $(\SL(3, \R), \mathbb{S}^2)$--structure. We denote a lift 
of $\Omega$ to $\mathbb{S}^2\subset \R^3$ by $\Omega^+.$
A \emph{light-cone representative} of $p \in \partial \Omega$ 
is a lift 
$v_p \in \mathcal L = \mathcal L^+ = \R^+ \cdot \partial \Omega^+.$
Each cusp $c$ of $S$ corresponds to an orbit of parabolic fixed 
points on $\partial \Omega.$ Choose an orbit representative
$p_c \in \partial \Omega,$ and hence a light-cone representative 
$v_{c} = v_{p_c} \in \mathcal L.$ The set 
$B = \{ \Gamma \cdot v_c \mid c \text{ is a cusp of } S\}$ 
is discrete. Let $C$ be the convex hull of $B.$ Then the 
projection of the faces of $\partial C$ onto $\Omega$ is 
a $\Gamma$--invariant ideal cell decomposition of $\Omega,$ 
and hence descends to an ideal cell decomposition of 
$\Omega/\Gamma,$ called an \emph{Epstein-Penner decomposition} 
by Cooper and Long. 

Varying the light-cone representatives 
$v_c$ gives a $(n-1)$--parameter family of 
$\Gamma$--invariant ideal cell decompositions of $\Omega$. 
Note that if each 
face of $C$ is a triangle, then a small perturbation 
of the lengths of the $p_c$ will not change the 
combinatorics of $C.$ Also, in the case of one cusp, varying the length 
of $p_c$ merely dilates $C$ and hence does not change 
the combinatorics of the convex hull. In particular, the decomposition of the surface 
$\Omega/\Gamma$ is canonical if $n = 1.$ However, if there is more than one 
cusp, then varying the length of just one $p_c$  
will eventually result in different decompositions, since it 
changes the relative heights of the vertices of $C.$

To summarise, given $p\in \widetilde{\Teich}_+(S_{g,n}),$ the convex 
hull construction by Cooper and Long~\cite{CL} associates 
to $p$ a canonical ideal cell decomposition $\Delta_p.$ 
Analogous to Penner~\cite{Penner1987}, 
define for any ideal cell decomposition 
$\Delta\subset S_{g,n}$ the sets
\begin{align*}
\mathring{\mathcal{C}}(\Delta) &= \{ p \in \widetilde{\Teich}_+(S_{g,n}) \mid \Delta_p = \Delta\},\\
{\mathcal{C}}(\Delta) &= \{ p \in \widetilde{\Teich}_+(S_{g,n}) \mid \Delta_p \subseteq \Delta\}.
\end{align*}
As in the classical case, we have 
${\mathcal{C}}(\Delta_1) \cap {\mathcal{C}}(\Delta_2) \neq \emptyset$ 
if and only if $\Delta_1\cap \Delta_2$ is an ideal cell 
decomposition of $S_{g,n},$ and in this case 
${\mathcal{C}}(\Delta_1) \cap {\mathcal{C}}(\Delta_2) = {\mathcal{C}}(\Delta_1\cap \Delta_2)$.
Moreover, if there is just one puncture, we may replace $\widetilde{\Teich}_+(S_{g,1})$ with ${\Teich}_+(S_{g,1})$ in the above definitions.

We can now state the main theorem of this paper:

\begin{theorem}\label{thm:main}
The set
\[
\{ \mathring{\mathcal{C}}(\Delta) \mid \Delta \text{ is an ideal cell decomposition of } S_{1,1}\}
\]
is an ideal cell decomposition of $\Teich_+(S_{1,1})$ 
that is invariant under the action of the mapping class group. Moreover, $\mathring{\mathcal{C}}$  is a natural bijection between the cells and the ideal cell decompositions of $S_{1,1}.$
\end{theorem}

The proof will be given in \S\ref{proof of main}, and the analogous statement for the thrice-punctured sphere is proved in  \S\ref{sec:S03}. First, some general results are developed in \S\ref{sec:projectivity}, the computation of holonomy is discussed in \S\ref{sec:monodromy}, and the coordinates for 
$\Teich_+(S_{1,1})$ are derived in \S\ref{sec:periphery}. Possible applications and further directions are discussed in \S\ref{sec:clovers} and \S\ref{sec:conclusion}.


\section{Projectivity}
\label{sec:projectivity}
Fock and Goncharov discovered in \cite{FG2006} that
the moduli space of mutually inscribed and
circumscribed triangles (from the perspective
of some affine patch in $\RP$) is naturally isomorphic
to the positive real line. We now develop an explicit formulation of this isomorphism, introducing the new concept of a \emph{(standard) trigonal matrix}. 

An element
of $\mathcal{P}^+_3$ is the projectivity
class of a combination of
three points in $\RP$ in general position
and three lines through those points such
that, from the perspective of some affine
patch, the triangle formed by the points
lies inside the trilateral formed by the
lines, as in the left of Figure \ref{fig:ayenay}.
\begin{figure}
\begin{center}
{\includegraphics{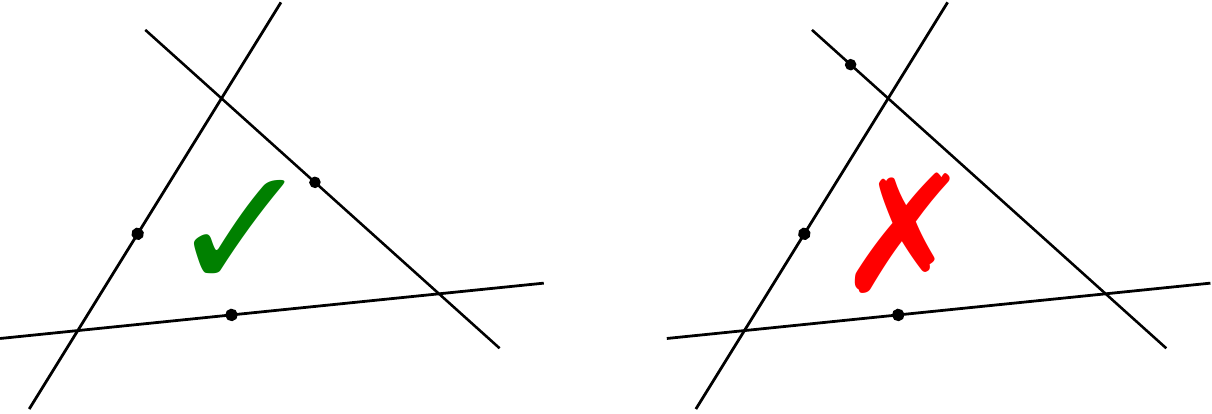}}
\end{center}
\caption{Representatives of an element and a non-element of $\mathcal{P}^+_3$.}\label{fig:ayenay}
\end{figure}
In terms more amenable to
calculation, such a triangle and trilateral are the projectivisations,
respectively, of a triple $(V_0,V_1,V_2)$ of vectors in $\R^3$
and a triple $(v_0,v_1,v_2)$ of covectors such that
$v_i.V_j \geq 0$, with equality only when $i = j$.

This is viewed as a pair $(V.\Delta,\Delta.v)$ of
left and right cosets of the subgroup $\Delta$
of diagonal matrices in $\GL(3, \R)$
admitting a representative $(V,v)$ such that
$v.V$ is a positive counter-diagonal
matrix, i.e. a matrix of the form
\[
\begin{pmatrix}
0 & + & +\\
+ & 0 & +\\
+ & + & 0
\end{pmatrix},
\]
where the $+$ entries are positive numbers, possibly different.

Let 
\begin{equation}\label{eqn:sigma}
\sigma = \begin{pmatrix} 0 & 1 & 0\\ 0 & 0 & 1\\ 1 & 0 & 0 \end{pmatrix}.
\end{equation}
If $M$ is positive counter-diagonal, then $\sigma.M.\sigma^{-1}$
is again positive counter-diagonal. The fixpoints of this
$\Z/3\Z$ action are called \emph{trigonal}. 
We will show that every element of $\mathcal{P}^+_3$
admits as representative a pair $(V.\Delta,\Delta.v)$ with $v.V$ being some
trigonal matrix. However, the space of trigonal matrices
is two-dimensional, and our intent is to show that $\mathcal{P}^+_3$
is one-dimensional. So we should like to produce,
for any such pair, a canonical trigonal matrix. One possible
choice of trigonal matrix (with one free parameter) is the \emph{standard trigonal} matrix
\begin{equation}\label{eq:cantri}
C^f_3 = 
\begin{pmatrix}
0 & f & 1\\
1 & 0 & f\\
f & 1 & 0
\end{pmatrix}.
\end{equation}

\begin{proposition}\label{prp:dblcos}
Every double coset of the form $\Delta P \Delta$ in $\GL(3,\R)$
with $P$ a positive counter-diagonal matrix admits
a unique standard trigonal representative.
\end{proposition}

\begin{proof}
Let $P$ be a positive counter-diagonal matrix in $\GL(3,\R)$.
We need to solve
\begin{equation}\label{eqn:pcdm}
C.P.D =
\begin{pmatrix}
0 & f & 1\\
1 & 0 & f\\
f & 1 & 0
\end{pmatrix}
\end{equation}
for $C,D$, and $f$. We should
end up with a parameter
space of solutions for $C$ and $D$, since we can scale
both $C$ and $D$ by scalars, but $f$
should be uniquely determined.

A direct calculation (see \S\ref{app:listings}, Listing~\ref{list:maxima0}) shows that
(\ref{eqn:pcdm}) admits solutions in $C,D,f$
precisely when
\[
f = \left ( \frac{P_{01}\cdot P_{12} \cdot P_{20}}
                 {P_{02}\cdot P_{10} \cdot P_{21}} \right ) ^ {1/3},
\]
where the $P_{i j}$ are the entries of $P.$
This concludes the proof of the proposition.
\end{proof}

Next, assuming that $v.V$ is standard trigonal, we wish to pick
$m \in \GL(3, \R)$ such that $v.m^{-1}$ and $m.V$ are
as nice as possible. To achieve this, we break duality between $v$ and $V$ here
and just let $m=V^{-1}$.
Even if $v.V$ is not assumed standard trigonal, we have:
\begin{theorem}\label{thm:proj}
Suppose $V,v \in \GL(3,\R)$ such that $v.V$ is
positive counter-diagonal. Let $m = (V.D)^{-1}$, where
$D$ is diagonal with $\lambda_0,\lambda_1,\lambda_2$
along its diagonal, and where
\begin{equation}\label{eqn:lambdas}
\lambda_0^3 = \frac{(v.V)_{1\,2} \cdot (v.V)_{2\,1}}
                    {(v.V)_{1\,0} \cdot (v.V)_{2\,0}},\qquad
\lambda_1^3 = \frac{(v.V)_{2\,0} \cdot (v.V)_{0\,2}}
                    {(v.V)_{2\,1} \cdot (v.V)_{0\,1}},\qquad
\lambda_2^3 = \frac{(v.V)_{0\,1} \cdot (v.V)_{1\,0}}
                    {(v.V)_{0\,2} \cdot (v.V)_{1\,2}}.
\end{equation}
Then $m.V.\Delta = \Delta$ and $\Delta.v.m^{-1} = \Delta.C_f^3$,
where $f$ is as in the proof of Proposition \ref{prp:dblcos} with
$P = v.V$.
\end{theorem}
\begin{proof}
By assumption, $v.V$ is positive counter-diagonal.
By Proposition \ref{prp:dblcos}, there exist
$C$ and $D$ such that $C.(v.V).D = (C.v).(V.D)$ are standard trigonal.
Then $m = (V.D)^{-1}$ is an element
of $\GL(3,\R)$ such that the image of
$(V,v)$ under $m$ and $(I_3,C_3^f)$ project
to the same configuration. This proves
the second portion of the theorem, reducing
our proof obligations to verifying that
$D$ as defined by (\ref{eqn:lambdas})
admits the existence of $C \in \Delta$ such
that $(C.v).(V.D)$ is standard trigonal. 
This can be verified by a direct calculation (see \S\ref{app:listings}, Listings~\ref{list:maxima1} and \ref{list:out1}).
\end{proof}


\section{Developing map and holonomy}
\label{sec:monodromy}
Let $S = S_{1,1}$ be a once-punctured torus with
a fixed ideal triangulation $\Delta_\ast$ and a fixed orientation. 
We slightly modify the framework of \cite{FG2006} 
to parameterise the marked strictly convex
projective structures with finite volume on $S.$ We also note that
Fock and Goncharov treat the more general space of framed
structures with geodesic boundary, of which the finite-volume
structures form a proper subset of positive codimension.
Lift $\Delta_\ast$ to the universal cover 
$\phi: \widetilde{S} \to S,$ denote the lifted triangulation by
$\widetilde{\Delta}_\ast$, and identify the group of deck transformations with $\pi_1(S).$
The developing map $\dev : \widetilde{S} \to \Omega$
for such a structure sends $\widetilde{\Delta}_\ast$
to an ideal triangulation $\tilde{\tau}$ of $\Omega$, 
which we may assume has straight edges (see \cite{WoTi}).
The edges
of this ideal triangulation have well-defined
endpoints on $\partial \Omega$, and the finite volume condition 
implies $\Omega$ is round, meaning
that every point on its boundary admits a unique 
tangent line (see \cite{CLT1}). Any triangle of $\widetilde{\Delta}_\ast$
therefore inherits an associated combination of
three points in $\RP$ and three lines through
these points such that, by strict convexity of $\Omega$,
the triangle formed by these points lies inside the
trilateral formed by the lines as in the left of Figure \ref{fig:ayenay}.

Associated to the geometric structure is a \emph{holonomy} $\hol: \pi_1(S)\to \SL(3,\R)$ 
which makes $\dev$ equivariant under the action of $\pi_1(S)$. 
That is, denoting $\Gamma = \hol(\pi_1(S))$, the map $\dev$
takes $\pi_1(S)$--equivalent triangles of $\widetilde{\Delta}_\ast$ to
$\Gamma$--equivariant triangles of $\tilde{\tau}$ in $\Omega$.
Finally, any two such developing maps for the same structure differ by a projectivity. In conclusion, then, for any triangle of $\Delta_\ast$, we get an associated element of $\mathcal{P}_3^+$, and
thereby, via Proposition~\ref{prp:dblcos} and Theorem~\ref{thm:proj},
a well-defined positive real number $f$. 
Likewise, to any oriented edge $\tilde{e}$ of $\widetilde{\Delta}_\ast$,
associate the pair $(\tilde{t},\tilde{t}')$ of triangles
in $\widetilde{\Delta}_\ast,$ where $\tilde{e}$ is adjacent
to both, and where the orientation of $\tilde{t}$ induces the orientation of $\tilde{e}$.
We may suppose that in an affine patch, the images
of $\tilde{t},\tilde{t}'$, and the 
flags attached to their vertices, look as shown in Figure~\ref{fig:nearedge}.

\begin{figure}[h]
\begin{center}
\includegraphics{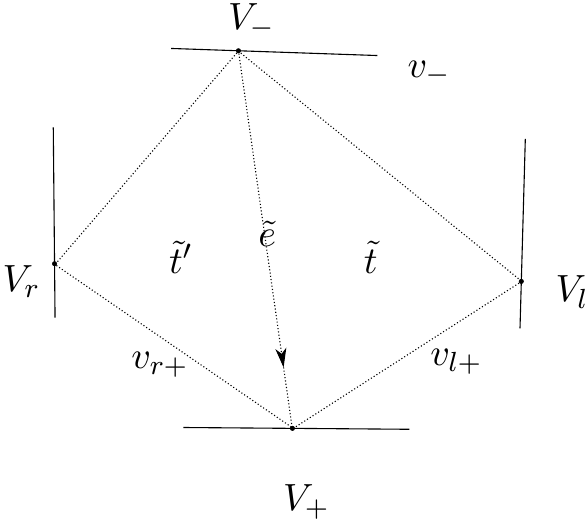}
\end{center}
\caption{Near the image of an oriented edge $\tilde{e}$: the labels match the orientation}\label{fig:nearedge}
\end{figure}

The flags depicted admit (nonzero) vector-covector representatives
$(V_i,v_i)$ for $i \in \{+,-,l,r\}$,
such that $v_i . V_i = 0$.
Let $v_{r+}$ be a covector representing the line through $V_+$ and $V_r$,
and likewise let $v_{l+}$ be a covector representing the
line through $V_+$ and $V_l$. Let
$v$ be the matrix whose rows are $v_-,v_{r+},v_{l+}$
and $V$ the matrix whose columns are
$V_-,V_r,V_l$. Letting $[x],[X]$ denote
the projections to $\RP$ of covectors $x$
and vectors $X$, we can associate
to this oriented edge $\tilde{e}$ the
triple of flags $(([v_-],[V_-]),([v_{r+}],[V_r]),([v_{l-}],[V_l]))$,
whose projectivity class is some element of $\mathcal{P}_3^+$,
to which we can associate a single, positive real number $f$ as in the section above.
(The \emph{triple ratio} of the triangle defined in \cite{FG2006} equals $f^3$.
However, our edge parameters' cubes are the \emph{reciprocals}
of Fock and Goncharov's edge parameters.)

We may then fix the developing map such that it has the following three properties:
\begin{enumerate}
\item the standard
basis for $\R^3$ projects to
the vertices of $\dev(\tilde{t})$;
\item $(1,0,0)^t$ and $(0,1,0)^t$
project to $V_-$ and $V_+$, respectively; and
\item
the kernels of the covectors $(0,t_{012},1)$,
$(1,0,t_{012})$, and $(t_{012},1,0)$
project to lines tangent to the boundary
of the convex domain $\Omega$ which is
the image of the developing map.
\end{enumerate}

This choice of developing map is unique up to isotopy,
given our choices of $\phi,\tilde{t},\tilde{e}$. 
Hence for each point in $\Teich_+(S)$ this fixes a unique developing map, and each such developing map determines a unique point in $\Teich_+(S)$ provided that the holonomy around the cusp is parabolic.

 \begin{figure}[h]
 \begin{center}
\includegraphics{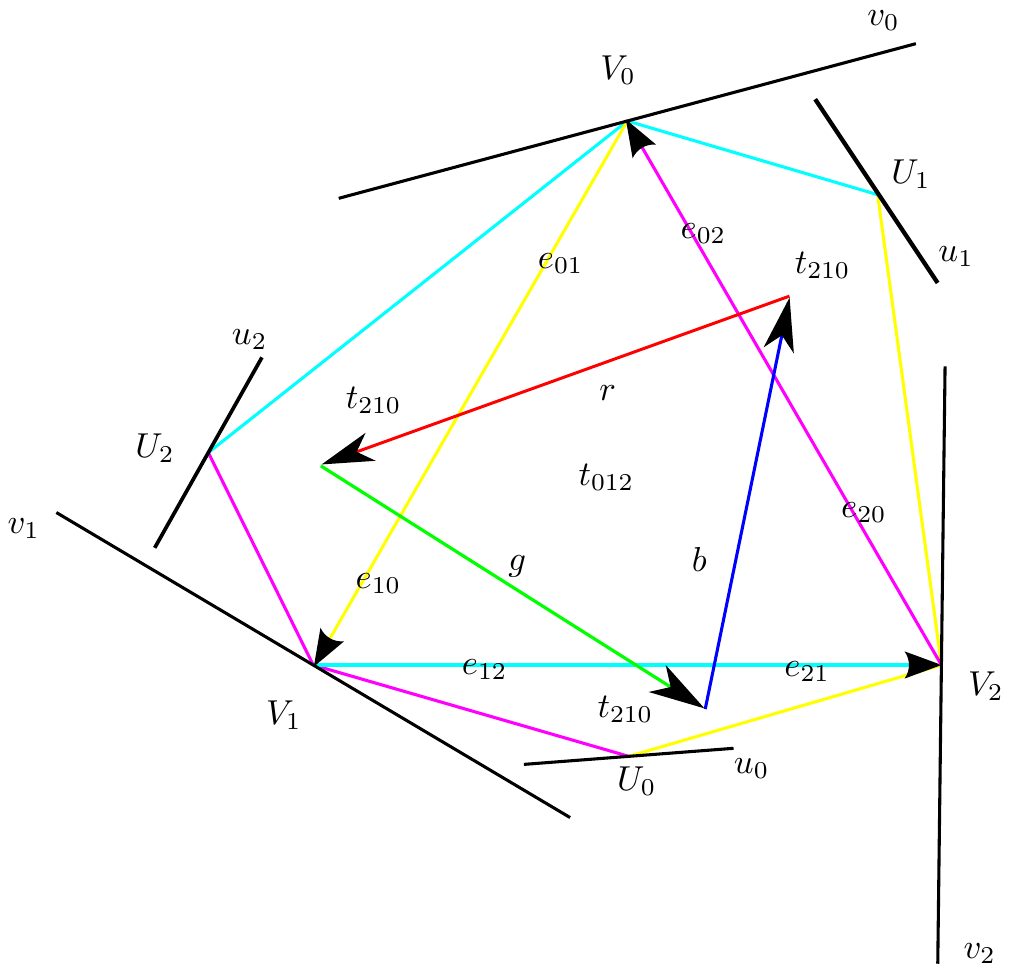}
\end{center}
\caption[Our standard developing map]{Our standard developing map with labels in our modified framework, where
$V_0 = \begin{pmatrix} 1 & 0 & 0 \end{pmatrix}^t$, 
$V_1 = \begin{pmatrix} 0 & 1 & 0 \end{pmatrix}^t$, 
$V_2 = \begin{pmatrix} 0 & 0 & 1 \end{pmatrix}^t$, and 
$v_0 = \begin{pmatrix} 0 &  t_{012} & 1 \end{pmatrix}$, 
$v_1 = \begin{pmatrix} 1 & 0 &  t_{012} \end{pmatrix}$, 
$v_2 = \begin{pmatrix} t_{012} & 1 & 0\end{pmatrix}.$
}
\label{fig:ourparams}
\end{figure}
 
 
\section{Periphery}
\label{sec:periphery} 
 
From the  developing map, we may now calculate the holonomy. To this end, it suffices to determine its values on generators of the fundamental group.
A marking of $\pi_1(S)$ is chosen as follows.
Adjacent to $\tilde{t}$ are three other
triangles which project to the same
triangle $\phi(\tilde{t}')$ of the ideal triangulation
of $S$. In the cyclic 
order induced by the orientation
of $\tilde{t}$, let these triangles be $c,m,y$, with
$y$ being the triangle adjacent to $\tilde{e}$.
Then we may choose for generators the deck
transformations $r,g,b$, where $r$ takes
$m$ to $y$, $g$ takes $y$ to $c$, and $b$ takes
$c$ to $m$.

The images of these deck transformations
now are simple to calculate; we just need
to calculate representatives $(V,v)$ as above
for $c,m,$ and $y$, and then use Theorem~\ref{thm:proj}
to get the holonomy. We can just
focus on $y$, and the other two will follow
by symmetry. 

We already know two flags of $y$: namely,
the first two standard basis vectors and
the first two associated covectors. 
To solve for the other vertex of $y$,
we use the edge parameters $e_{01}$ and $e_{10}$.
To solve for the associated line through
this vertex, we use the other face parameter and the definition of these parameters as triple ratios.
Solving for some element of the vertex of $y$ (regarding the vertex as
a one-dimensional subspace of $\R^3$) gives the vertex $U_2$ (see \S\ref{app:listings}, Listing~\ref{list:maxima6}, for the computation):
\[
U_2 = \left \langle \begin{pmatrix} (e_{1 0}^3 + 1) \cdot t_{0 1 2}^2\\
                                    (e_{0 1}^3 + 1) \cdot e_{1 0}^3\\
                                    -e_{1 0}^3 \cdot t_{0 1 2}
                    \end{pmatrix} \right \rangle.
\]

By symmetry (or by independent calculations), we conclude
\[
U_0 = \left \langle \begin{pmatrix} -e_{2 1}^3 \cdot t_{0 1 2}\\
                                    (e_{2 1}^3 + 1) \cdot t_{0 1 2}^2\\
                                    (e_{1 2}^3 + 1) \cdot e_{2 1}^3
                    \end{pmatrix} \right \rangle,
\qquad \qquad
U_1 = \left \langle \begin{pmatrix} (e_{2 0}^3 + 1) \cdot e_{0 2}^3\\
                                    -e_{0 2}^3 \cdot t_{0 1 2}\\
                                    (e_{0 2}^3 + 1) \cdot t_{0 1 2}^2
                    \end{pmatrix} \right \rangle.
\]

We have now found the other vertices of our configurations $c,m,y$,
so we have all the triangle parts. Care has to be taken to keep the vertices
in a consistent order, so that the holonomy comes out without any extraneous rotation.
(Indeed, the holonomy for $S_{0,3}$ differs only in this respect.)

We next need to compute covector representatives of the
lines through the $U$-vertices. We solve for these
using the other face parameter $t_{210}$ (see \S\ref{app:listings}, Listings~\ref{list:maxima9} and \ref{list:maxima10}).
This gives:
\[
u_2 = \left \langle \begin{pmatrix} e_{0 1}^3 \cdot e_{1 0}^3\ &
                      t_{0 1 2}^2 \cdot t_{2 1 0}^3\ &
                      t_{0 1 2} \cdot \left ( e_{0 1}^3 \cdot t_{2 1 0}^3 +
                                              t_{2 1 0}^3 + e_{0 1}^3 \cdot e_{1 0}^3 + e_{0 1}^3
                                      \right ) \end{pmatrix} \right \rangle
\]
and again, by symmetry or independent computation, we conclude
\[
u_0 = \left \langle \begin{pmatrix}
                      t_{0 1 2} \cdot \left ( e_{1 2}^3 \cdot t_{2 1 0}^3 +
                                              t_{2 1 0}^3 + e_{1 2}^3 \cdot e_{2 1}^3 + e_{1 2}^3
                                      \right )\ &
                      e_{1 2}^3 \cdot e_{2 1}^3\ &
                      t_{0 1 2}^2 \cdot t_{2 1 0}^3 \end{pmatrix} \right \rangle
\]
and
\[
u_1 = \left \langle \begin{pmatrix}
                     t_{0 1 2}^2 \cdot t_{2 1 0}^3\ &
                     t_{0 1 2} \cdot \left ( e_{2 0}^3 \cdot t_{2 1 0}^3 +
                                             t_{2 1 0}^3 + e_{2 0}^3 \cdot e_{0 2}^3 + e_{2 0}^3
                                     \right )\ &
                     e_{2 0}^3 \cdot e_{0 2}^3 \end{pmatrix} \right \rangle.
\]

Recall that $r$ takes $m$ to $y$, $g$ takes $y$ to $c$, and $b$ takes $c$ to $m$.
So with the above results, we can define the trilaterals, completing our
construction of the configurations we need for
the monodromy calculation (see \S\ref{app:listings}, Listing~\ref{list:maxima11}).
The formul\ae\ for $r$, $g$, and $b$ are complicated
and not particularly illuminating,
so we leave them in the internals of the computer.

From the above, the fundamental group of $S = S_{1,1}$ has  
the presentation $\langle r,g,b\ |\ b\cdot g \cdot r\rangle$ and a fixed marking.
The element $r\cdot g \cdot b$ is a peripheral
element, representing a simple closed loop around the cusp. As stated in
\S\ref{sec:Convex hull constructions}, finite volume requires this
peripheral element to have parabolic holonomy.
This means that its characteristic
polynomial is of the form $k\cdot(\lambda - 1)^3$. We
may calculate the conditions this imposes on the
Fock-Goncharov parameters using the characteristic polynomial. 
A direct computation (see \S\ref{app:listings}, Listing~\ref{list:maxima12}) shows that 
the characteristic polynomial of $r\cdot g\cdot b$ is proportional to
$(\lambda - T^3)\cdot(E^3 \cdot \lambda - T^3)\cdot(T^6 \cdot \lambda - E^3),$
where $T$ is the product of the face parameters and $E$ is the product of the
edge parameters. We therefore have: 
\begin{lemma}\label{lem:parhol}
A strictly convex projective structure
on $S_{1,1}$ has parabolic peripheral holonomy
(and finite volume) if and only if
the product of the face parameters and
the product of the edge parameters both equal 1.
\end{lemma}

To sum up, we now have an identification of $\Teich_+(S_{1,1})$ with 
\[
 \{ (t_{012}, t_{210}, e_{01}, e_{10}, e_{02}, e_{20}, e_{12}, e_{21}) \in \R^8_+ \mid t_{012}t_{210}=1, \ e_{01} e_{10} e_{02} e_{20} e_{12} e_{21}=1 \}.
\]


\section{Cells}
\label{sec:canonicity}
An algorithm to compute the ideal cell decompositions of Cooper and Long~\cite{CL} 
was recently described by Tillmann and Wong~\cite{WoTi}, based on an algorithm for hyperbolic surfaces by Weeks~\cite{MR1241189}.
This algorithm takes as starting point a fixed ideal 
triangulation of $S_{g,n}$ and then computes the 
canonical ideal cell decomposition associated to a point 
in moduli space using an edge flipping algorithm 
followed by possibly deleting redundant edges. This 
allows one to keep track of the marking, the isotopy class 
of every intermediate ideal triangulation, and the 
isotopy class of the final ideal cell decomposition.

We make use of a portion of this work as follows. 
Let $p=(t_{012}, t_{210}, e_{01}, e_{10}, e_{02}, e_{20}, e_{12}, e_{21}) \in \Teich_+(S_{1,1})$
and choose a light-cone representative of the cusp; we pick not
a standard basis vector, but the slightly different vector
\[
S_0 = \begin{pmatrix} e_{2 0} \cdot e_{0 2} \cdot e_{2 1} \\ 0 \\ 0 \end{pmatrix}.
\]
This represents the image of the terminal endpoint of
$\tilde{e}$ as well as the standard basis vector does.
The orbit of this vector under the holonomy of
the fundamental group is some collection of vectors,
including the elements
\[
S_1 = b^{-1}.S_0 = \begin{pmatrix} 0 \\ e_{0 1} \cdot e_{1 0} \cdot e_{0 2} \\ 0 \end{pmatrix},
\qquad \qquad
S_2 = g.S_0 = \begin{pmatrix} 0 \\ 0 \\ e_{1 2} \cdot e_{2 1} \cdot e_{1 0} \end{pmatrix}
\]
which represent the other vertices. Note that $S_1$ represents
the initial endpoint of $\tilde{e}$. Let $\omega$ be the covector such
that $\omega.S_0 = \omega.S_1 = \omega.S_2 = 1$. Under
our assumption of parabolic holonomy, we may write
\[
\omega = \begin{pmatrix} e_{0 1} \cdot e_{1 0} \cdot e_{1 2} &
                         e_{1 2} \cdot e_{2 1} \cdot e_{2 0} &
                         e_{2 0} \cdot e_{0 2} \cdot e_{0 1} \end{pmatrix}.
\]
That is, the plane $P$ through $S_0,S_1,S_2$ is given as $P = \{v : \omega.v = 1 \}$.

Then $p \in {\mathcal{C}}(\Delta_\ast)$ if and only if for every $\gamma\in \Gamma$ we have
$\omega.(\gamma.S_0) \geq 1$, i.e.\thinspace when every element
of the orbit of $S_0$ does not lie on the same side of $P$ as the origin. It was shown in 
\cite{WoTi} that it suffices to show this locally, thus turning it into a finite problem.
 For our purposes, this means that
$p \in {\mathcal{C}}(\Delta_\ast)$ is equivalent to showing that for all $v \in \{r.S_0, g.S_1, b.S_2\}$,
$\omega.v \geq 1$. By construction, $r.S_0 = g^{-1}.S_1$.
Therefore call the quantity $\omega.(r.S_0)$ \emph{yellow bending};
we denote yellow bending by $YB$.\footnote{We wish to emphasise that we make the natural choice of using a two letter acronym here.} 
We call the condition $YB \geq 1$ \emph{yellow consistency}.
In the case of yellow consistency, the convex hull is non-concave along the associated edge.
We call the condition $YB = 1$ \emph{yellow flatness}.
We make similar definitions for cyan and magenta, with 
$CB$ and $MB$ being their associated bendings. 

Note that if one deletes $\phi(\tilde{e})$
from $\Delta_\ast$ to get an ideal quadrilation
$\Delta'$ of $S$, then for all points $p$ in moduli
space, $YB(p) = 1$ is equivalent to
$p \in \mathcal{C}(\Delta')$. Likewise for the other
edges.

Consistency occurs when all bendings
are greater than or equal to one.
Now, the bendings are all rational
functions, so consistency is a semi-algebraic
condition. To show that the set of
canonical structures is a cell, we will
show the following.

\begin{lemma}\label{lem:flatdisc}
The semi-algebraic set determined by the cyan
flatness condition is a smooth, properly embedded cell of codimension 1
in $\Teich_+(S_{1,1})$.
\end{lemma}
\begin{lemma}\label{lem:flatdisjoint}
The cyan, yellow and magenta flatness conditions are pairwise disjoint.
\end{lemma}

Assuming the lemmata, we can now prove the main theorem:

\begin{proof}[Proof of Theorem \ref{thm:main}]\label{proof of main}
Using our modification of Fock and Goncharov's coordinates  
as described above, $\Teich_+(S_{1,1})$ is identified 
with a properly embedded 6--disc in the positive
orthant of $\R^8$. Since the action of the mapping 
class group of $S_{1,1}$ is transitive on the set of 
all ideal triangulations of $S_{1,1},$ as well as on the 
set of all ideal quadrilations of $S_{1,1},$ it suffices 
to show that
\begin{enumerate}
\item one of the sets $\mathring{\mathcal{C}}(\Delta),$ 
where $\Delta$ is an arbitrary but fixed ideal 
triangulation, is an ideal cell; and 
\item one of the sets $\mathring{\mathcal{C}}(\Delta'),$ 
where $\Delta'$ is an arbitrary but fixed ideal 
quadrilation, is an ideal cell.
\end{enumerate}
The latter is the contents of Lemma~\ref{lem:flatdisc} 
for $\Delta'$ the ideal quadrilation obtained from $\Delta_\ast$ by 
deleting the cyan edge. Hence we turn to the former.
Let $\Delta_0,$ $\Delta_1$ and $\Delta_2$ denote the 
three ideal quadrilations obtained by deleting one of 
the three ideal edges from $\Delta_\ast.$ Then the 
frontier of $\mathring{\mathcal{C}}(\Delta)$ is 
contained in 
$\mathring{\mathcal{C}}(\Delta_0)\cup \mathring{\mathcal{C}}(\Delta_1)\cup \mathring{\mathcal{C}}(\Delta_2).$

For each ideal quadrilation $\Delta_i$, it follows from 
Lemma~\ref{lem:flatdisc} that 
$\mathring{\mathcal{C}}(\Delta_i)$ 
is a smooth properly embedded 5-disc. Whence each 
$\mathring{\mathcal{C}}(\Delta_i)$ 
cuts $\Teich_+(S_{1,1})$ into two 6--discs. Now since 
any two $\mathring{\mathcal{C}}(\Delta_i)$ and 
$\mathring{\mathcal{C}}(\Delta_j)$ are disjoint by 
Lemma~\ref{lem:flatdisjoint} it follows that 
$\Teich_+(S_{1,1}) \setminus \bigcup_i \mathring{\mathcal{C}}(\Delta_i)$ 
consists of four open 6--discs, and that the one  6--disc with all three 5--discs in its boundary is
$\mathring{\mathcal{C}}(\Delta).$ 

The statement that $\mathring{\mathcal{C}}$  is a bijection now follows from the well-known fact that any two ideal triangulations of $S$ are related by a finite sequence of edge flips.
\end{proof}

\begin{proof}[Proof of Lemma \ref{lem:flatdisc}]
We can
solve $CB = 1$ for $e_{2 0}.$ (See \S\ref{app:listings}, Listing~\ref{list:maxima16}.) 
Let
\begin{align}\label{eqn:topc}
\top_c &= e_{0 1} \cdot e_{1 0}^2 \cdot e_{1 2}^2 \cdot e_{2 1} \cdot t_{0 1 2} 
         - e_{1 2}^3 - 1,\\
\label{eqn:botc}
\bot_c &= e_{2 1}^3 \cdot t_{0 1 2} + t_{0 1 2} 
         - e_{0 1}^2 \cdot e_{1 0} \cdot e_{1 2} \cdot e_{2 1}^2.
\end{align}
Then
\[
(CB = 1)\ \equiv\ (e_{1 0}\cdot e_{1 2}^2 \cdot t_{0 1 2} \cdot \bot_c \cdot e_{2 0}
                = e_{2 1} \cdot \top_c).
\]
So when we project the subset $\mathcal{C}$ cut
out by cyan flatness onto the $t_{0 1 2}, e_{0 1}, e_{1 0}, e_{1 2}, e_{2 1}$
plane, the image decomposes into three pieces:
\[
(\bot_c > 0 \land \top_c > 0) \lor 
(\bot_c = 0 \land \top_c = 0) \lor
(\bot_c < 0 \land \top_c < 0),
\]
with the fiber over every point of the
first and last pieces a single point
(since we get the graph of 
$e_{2 0} = e_{2 1} \cdot \top_c / (e_{1 0} \cdot e_{1 2}^2 \cdot t_{0 1 2} \cdot \bot_c)$
over these regions), and the fiber over a point in the
middle piece a whole $\R_+$.

We will show that the first and last
pieces are 5-discs, and that the middle
piece is a 3-disc. Then $CB = 1$ will
be the union of two 5-discs and a
4-disc (the product of the middle
3-disc with $\R_+$) in their boundaries; this union is
again a 5-disc.

Now, if $\sim$ is one of $<,>,=$, then
\[
(\top_c \sim 0) \equiv t_{0 1 2} 
                       \sim \frac{e_{1 2}^3 + 1}
                                 {e_{0 1} \cdot e_{1 0}^2 \cdot e_{1 2}^2 \cdot e_{2 1}},
\qquad \qquad
(\bot_c \sim 0) \equiv t_{0 1 2}
                       \sim \frac{e_{0 1}^2 \cdot e_{1 0} \cdot e_{1 2} \cdot e_{2 1}^2}
                                 {e_{2 1}^3 + 1}.
\]
Let $p = (e_{1 2}^3 + 1)/(e_{0 1} \cdot e_{1 0}^2 \cdot e_{1 2}^2 \cdot e_{2 1})$
and $q = e_{0 1}^2 \cdot e_{1 0} \cdot e_{1 2} \cdot e_{2 1}^2/(e_{2 1}^3 + 1)$.
Let $\uparrow$ and $\downarrow$ denote
 the maximum and minimum operators respectively.
The first piece is equivalent to $t_{0 1 2} > p \uparrow q$, which is
the region above the graph of a function (viz. $p \uparrow q$); this region
is a 5-disc.

Now, $p$ and $q$ are both positive functions, assuming
their arguments are positive. So the last piece is
the region between the graph of a positive function (viz. $p \downarrow q$) and the
$e_{0 1},e_{1 0},e_{1 2},e_{2 1}$-plane; this is again just a 5-disc.

For the middle piece, one sees that
$p = q$ is equivalent to
\[
e_{0 1}^3 = \frac{(e_{1 2}^3 + 1) \cdot (e_{2 1}^3 + 1)}
                 {e_{1 0}^3 \cdot e_{1 2}^3 \cdot e_{2 1}^3},
\]
which is the graph of a positive function
(the cube root of the right-hand side) over $\R^3_+,$ which is a 3-disc.
\end{proof}
\begin{proof}[Proof of Lemma \ref{lem:flatdisjoint}]
To show disjointness, it will suffice to show that
the projections of $\mathcal{C}(Q_c)$ and $\mathcal{C}(Q_y)$
to the $t_{0 1 2},e_{0 1},e_{1 0},e_{1 2},e_{2 1}$-plane
are disjoint, where $Q_e$ is the ideal quadrilation
obtained from $\Delta_\ast$ by forgetting $e$.
To that end, we repeat the computation from the proof of the previous lemma for 
yellow flatness (see \S\ref{app:listings}, Listing~\ref{list:maxima17}).
Based on this, define
\begin{align*}
\top_y &= e_{0 1}\cdot e_{1 0}^3 \cdot e_{1 2}^2 \cdot e_{2 1} \cdot t_{0 1 2}
        +e_{0 1}\cdot e_{1 2}^2 \cdot e_{2 1} \cdot t_{0 1 2} - e_{1 0},\\
\bot_y &= e_{0 1}\cdot t_{0 1 2} - e_{0 1}^3 \cdot e_{1 0} \cdot e_{1 2}\cdot e_{2 1}^2
        -e_{1 0}\cdot e_{1 2} \cdot e_{2 1}^2.
\end{align*}
The projection of the cyan flat-set 
to this plane decomposes into
the pieces shown earlier, and
the projection of the yellow flat-set
decomposes likewise.

Algorithms for cylindrical algebraic
decomposition can return a list
containing a point from every cell 
of this decomposition. We may run
such an algorithm in {\tt Sage}\rm~\cite{sage} on the
intersection of the cyan and yellow
flat-set projections (see \S\ref{app:listings}, Listings~\ref{list:SAGE} and \ref{list:SAGEout}).
The computation used \texttt{qepcad}; and the output is
a list of a point from every cell in the intersection
of the projections to the $t_{0 1 2},e_{0 1},e_{1 0},e_{1 2},e_{2 1}$-plane
of the cyan and yellow flat-sets.
But this list is empty; therefore, their intersection is empty.
\end{proof}


\section{The thrice-punctured sphere}
\label{sec:S03}

An ideal triangulation of the thrice-punctured sphere also consists of three properly embedded arcs, and hence divides the sphere into two ideal triangles. So there are six edge invariants and two triangle invariants. However, as there are three cusps, there are three holonomy conditions to ensure that the peripheral elements are parabolic. Using the same set-up as in Figure~\ref{fig:ourparams}, but taking into account that each of the three indicated deck transformations now fixes the respective vertex of the triangle, a direct computation yields an identification of $\Teich_+(S_{0,3})$ with the set
\[
 \{ (t_{012}, t_{210}, e_{01}, e_{10}, e_{02}, e_{20}, e_{12}, e_{21}) \in \R^8_+ \mid t_{012}t_{210}=1, \ e_{01}= \frac{1}{e_{10}}= e_{12}=\frac{1}{e_{21}} =  e_{20} = \frac{1}{e_{21}} \},
\]
showing that $\Teich_+(S_{0,3})$ is 2--dimensional as proven by Marquis~\cite{Marq1}. We will simply write $(t_{012}, e_{01}) \in \Teich_+(S_{0,3}).$ The result of the convex hull construction of Cooper and Long~\cite{CL} now depends, for each point in $\Teich_+(S_{0,3}),$ on the lengths of the light-cone representatives for the three cusps, up to scaling all of them by the same factor. Whence there is an ideal cell decomposition of $S_{0,3}$ associated to each point in the \emph{decorated} moduli space $\widetilde{\Teich}_+(S_{0,3}),$ and the latter can be identified with the positive orthant in $\R^5.$

\begin{figure}[h]
	\centering	
 				\includegraphics[width=5.5cm]{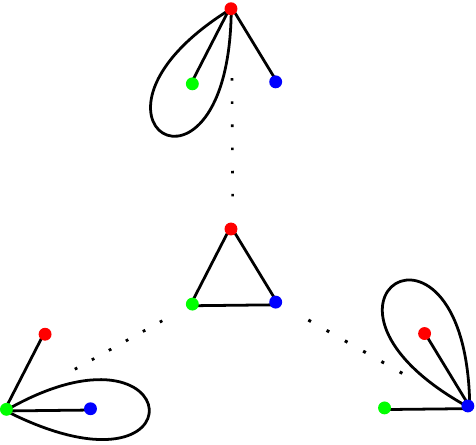}
		\caption{Flip graph of the thrice-punctured sphere}
	\label{fig:arcpod}
\end{figure}

Using Alexander's trick, it is an elementary exercise 
to determine that there are exactly four ideal 
triangulations and three ideal quadrilations of 
$S_{0,3}$. The \emph{flip graph} is the tripod shown 
in Figure~\ref{fig:arcpod}, where the quadrilations are 
obtained as intersections of the two pictures at the 
ends of each dotted arc. Let $\Delta_\ast$ be the ideal 
triangulation with the property that there is an arc 
between any two punctures. The punctures are labelled 
by $0,$ $1$ and $2,$ and $\Delta_i$ is the triangulation 
obtained from $\Delta\ast$ by performing an edge flip on 
the edge not meeting $i \in \{0, 1, 2\}.$ Moreover, the 
mapping class group of $S_{0,3}$ is naturally isomorphic 
with the group of all permutations of the three cusps.

\begin{theorem}\label{thm:main2}
The set
\[
\{ \mathring{\mathcal{C}}(\Delta) \mid \Delta \text{ is an ideal cell decomposition of } S_{0,3}\}
\]
is an ideal cell decomposition of $\widetilde{\Teich}_+(S_{0,3})$ 
that is invariant under the action of the mapping class group. Moreover, $\mathring{\mathcal{C}}$  is a natural bijection between the cells and the ideal cell decompositions of $S_{0,3}.$
\end{theorem}

\begin{proof}
Using the above set-up, let $(e_{01}, t_{012}) \in \Teich_+(S_{0,3}),$ and 
choose light-cone representatives $V_0 = \begin{pmatrix} \frac{1}{\omega_0} \\ 0 \\ 0 \end{pmatrix}$, 
$V_1 = \begin{pmatrix} 0 \\ \frac{1}{\omega_1} \\ 0 \end{pmatrix}$, 
$V_2 = \begin{pmatrix} 0 \\ 0 \\ \frac{1}{\omega_2} \end{pmatrix}.$ 
The corresponding point in $\widetilde{\Teich}_+(S_{0,3})$ is identified with 
$(t_{012}, e_{01}, \omega_0, \omega_1, \omega_2) \in \R^5_+.$

The three convexity conditions associated to the edges of $\Delta_\ast$ are equivalent to
\begin{align*}
\omega_1 - \omega_2\; t_{012} + \omega_0\; t^2_{012} &\ge 0,\\
\omega_2 - \omega_0\; t_{012} + \omega_1\; t^2_{012} &\ge 0,\\
\omega_0 - \omega_1\; t_{012} + \omega_2\; t^2_{012} &\ge 0.
\end{align*}
It is interesting to note that these are independent of $e_{01},$ 
whence the decomposition is the product of a decomposition of 
$\R^4$ with $\R.$. These define $\mathcal{C} (\Delta_\ast).$
Using these conditions, it is straight forward to check that 
${\mathcal{C}}(\Delta_\ast)$ is an ideal cell. Moreover, 
\[
\mathring{\mathcal{C}}(\Delta_0) = \{ (t_{012}, e_{01}, \omega_0, \omega_1, \omega_2) \in \R^5_+ \mid
\omega_2 - \omega_0\; t_{012} + \omega_1\; t^2_{012}<0\},
\]
and likewise for $\Delta_1$ and $\Delta_2$ by cyclically permuting the appropriate subscripts. This divides $\R^5_+$ into four open 5--balls along three properly embedded open 4--balls, and the dual skeleton to this decomposition is naturally identified with the flip graph of $S_{0,3}.$ The invariance by the action of the mapping class group follows since it acts as the group of permutations on $\{0, 1, 2\}.$
\end{proof}


\section{Cloverleaf patches}
\label{sec:clovers}

Of course, we would now like to explore how the geometry of a marked strictly convex projective structure relates to its relative position within a cellular subsets of moduli space, and how the geometry varies as one moves around in moduli space or towards the ``boundary," i.e.\thinspace as at least one coordinate becomes very small or very large.
Structures near the boundary might be difficult to
draw properly, given an arbitrary affine patch.
The image $\Omega$ of the developing map might go off to
infinity or become thin, yielding uninformative pictures.
We must take care, then, with our choice of affine patch.
One such choice is \emph{Benz\'ecri position}, 
discussed in \cite{Benz, CLT2}. Our investigations here
led us to discover another such choice,
which we call a \emph{cloverleaf patch}.

\label{def:cloverleaf}
Let $S$ be an oriented marked convex projective
surface with geodesic boundary. Let $\Delta_\ast$
be an ideal triangulation of $S$. Let $\tilde{t}$
and $\tilde{e}$ be an adjacent triangle and
edge in the lift $\widetilde{\Delta_\ast}$
of $\Delta_\ast$ through a universal cover $\phi$.

Let $\dev$ be the developing map satisfying our choices above, viz.
\begin{enumerate}
\item the standard
basis for $\R^3$ projects to
the image under $\dev$
of the vertices of $\tilde{t}$;
\item $(1,0,0)^t$ and $(0,1,0)^t$
project to $V_-$ and $V_+$, respectively; and
\item
the kernels of the covectors $(0,t_{012},1)$,
$(1,0,t_{012})$, and $(t_{012},1,0)$
project to lines tangent to the boundary
of the convex domain $\Omega$ that is
the image of $\dev$,
where $t_{012}$ is the parameter associated to $\tilde{t}$.
\end{enumerate}
Finally, let 
\[
\omega = \begin{pmatrix} e_{0 1} \cdot e_{1 0} \cdot e_{1 2} &
                         e_{1 2} \cdot e_{2 1} \cdot e_{2 0} &
                         e_{2 0} \cdot e_{0 2} \cdot e_{0 1} \end{pmatrix}
\]
and let $P$ be the affine patch given by $\omega.v = 1$.

Then $P$ is the \emph{cloverleaf patch} of the structure
on $S$ with respect to $\tilde{t}$ and $\tilde{e}$.

\begin{theorem}\label{thm:cloverleaf}
Let $\Omega'$ be the image of $\Omega$ under
the affine isomorphism $\alpha$ between a cloverleaf
patch and $\R^2 = \mathbb{C}$ that
sends the vertices of $\tilde{t}$ to
the cube roots of unity and the vertices
of $\tilde{e}$ to the primitive cube
roots of unity.

Then $\Omega'$ contains the triangle
whose vertices are the cube roots of
unity, and is contained in the union
of unit discs centered at the cube
roots of $-1.$
\end{theorem}

\begin{figure}
\begin{center}
\scalebox{0.5}{\includegraphics{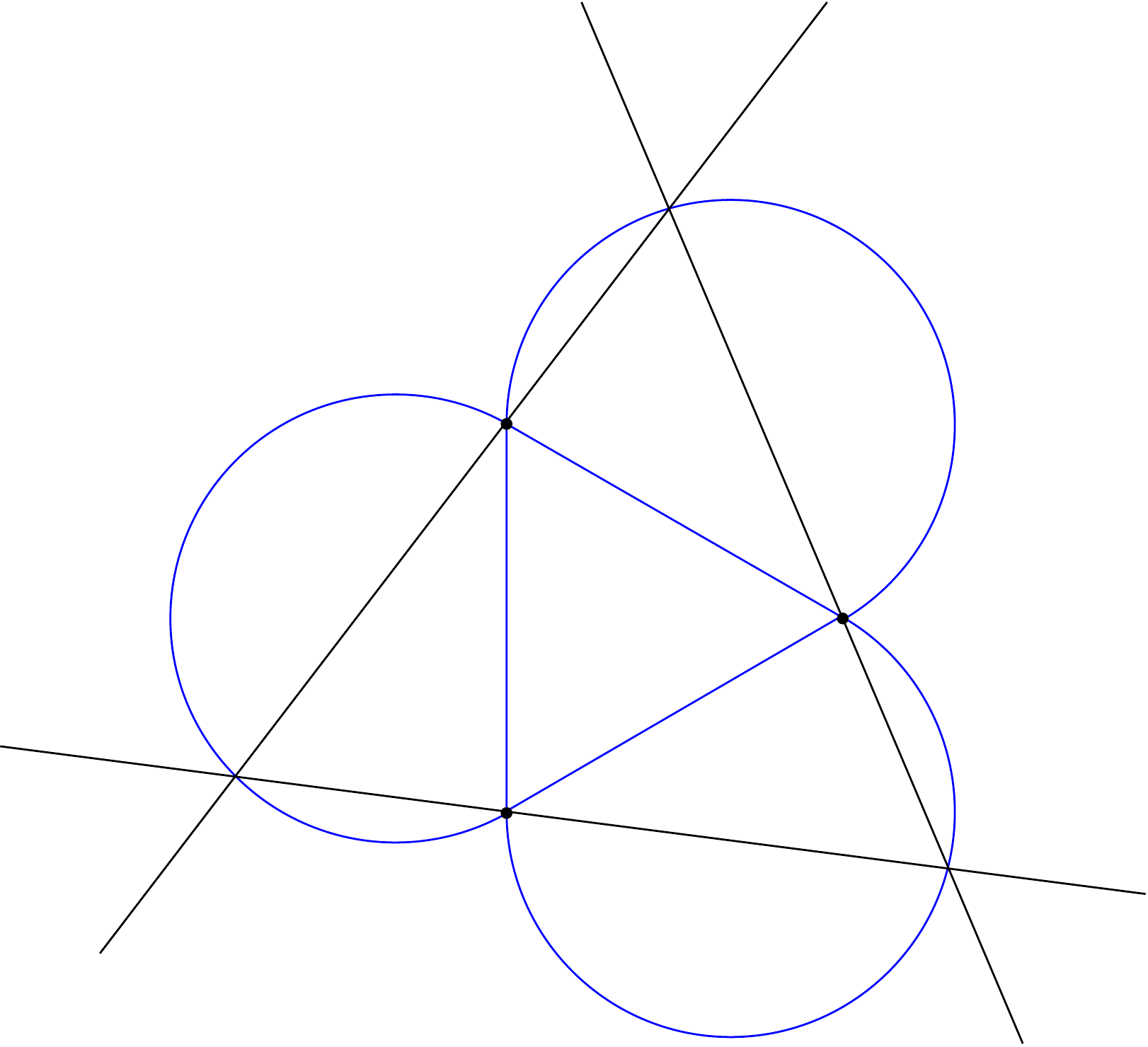}}
\end{center}
\caption{Three vertices and supporting hyperplanes in a cloverleaf patch.}\label{fig:cloverleaf}
\end{figure}

\begin{proof}
The domain $\Omega'$ by definition contains the
image of $\tilde{t}$, and the image
of $\tilde{t}$ under $\alpha$ is,
by definition, the triangle whose
vertices are the roots of unity. This
concludes the proof of the first claim.

Let
\[
\rho =
\begin{pmatrix}
\frac{e_{20} \cdot e_{02}}
     {e_{12} \cdot e_{10}} & 0 & 0\\
0 & \frac{e_{01} \cdot e_{10}}
         {e_{20} \cdot e_{21}} & 0\\
0 & 0 & \frac{e_{12} \cdot e_{21}}
             {e_{01} \cdot e_{02}}
\end{pmatrix},
\]
and let $\sigma$ be as in (\ref{eqn:sigma}).
Then $\sigma . \rho$ is an element of 
$\SL(3, \R)$ which permutes the
vertices of the image of $\tilde{t}$.
But it also permutes the covectors
$v_0 = \begin{pmatrix} 0 &  t_{012} & 1 \end{pmatrix}$, 
$v_1 = \begin{pmatrix} 1 & 0 &  t_{012} \end{pmatrix}$, and
$v_2 = \begin{pmatrix} t_{012} & 1 & 0\end{pmatrix}$.
We've chosen our developing map so that
the kernels of these covectors project to
tangent lines to $\partial \Omega$ at the
vertices of the image of $\tilde{t}$.
So $\Omega$ lies within
the trilateral 
$\tau = \{ V : \langle \forall i \in \{0,1,2\} : v_i.V > 0\rangle\}$.
This trilateral is symmetric under $\sigma . \rho$,
so its image is symmetric under $s = \alpha \circ \sigma . \rho \circ \alpha^{-1}$.
But $s$ is an order 3 affine automorphism of $\R^2$
permuting the image of the vertices of $\tilde{t}$, the
cube roots of unity. So $s$ is just $2\pi/3$ rotation
about the origin. Therefore the image of
$\tau$ is a trilateral with an order 3 rotational
symmetry. Hence it is an equilateral trilateral
circumscribed about the equilateral triangle
formed by the roots of unity.

Using elementary Euclidean geometry, it is
easy to see that this trilateral must
lie in the region described, the
union of the unit discs centered
at cube roots of $-1.$ (See Figure~\ref{fig:cloverleaf}.)
Since $\Omega'$ lies inside this trilateral,
it lies inside the region as well.
\end{proof}

\begin{figure}[h]
	\centering	
 		\includegraphics[width=3.4cm]{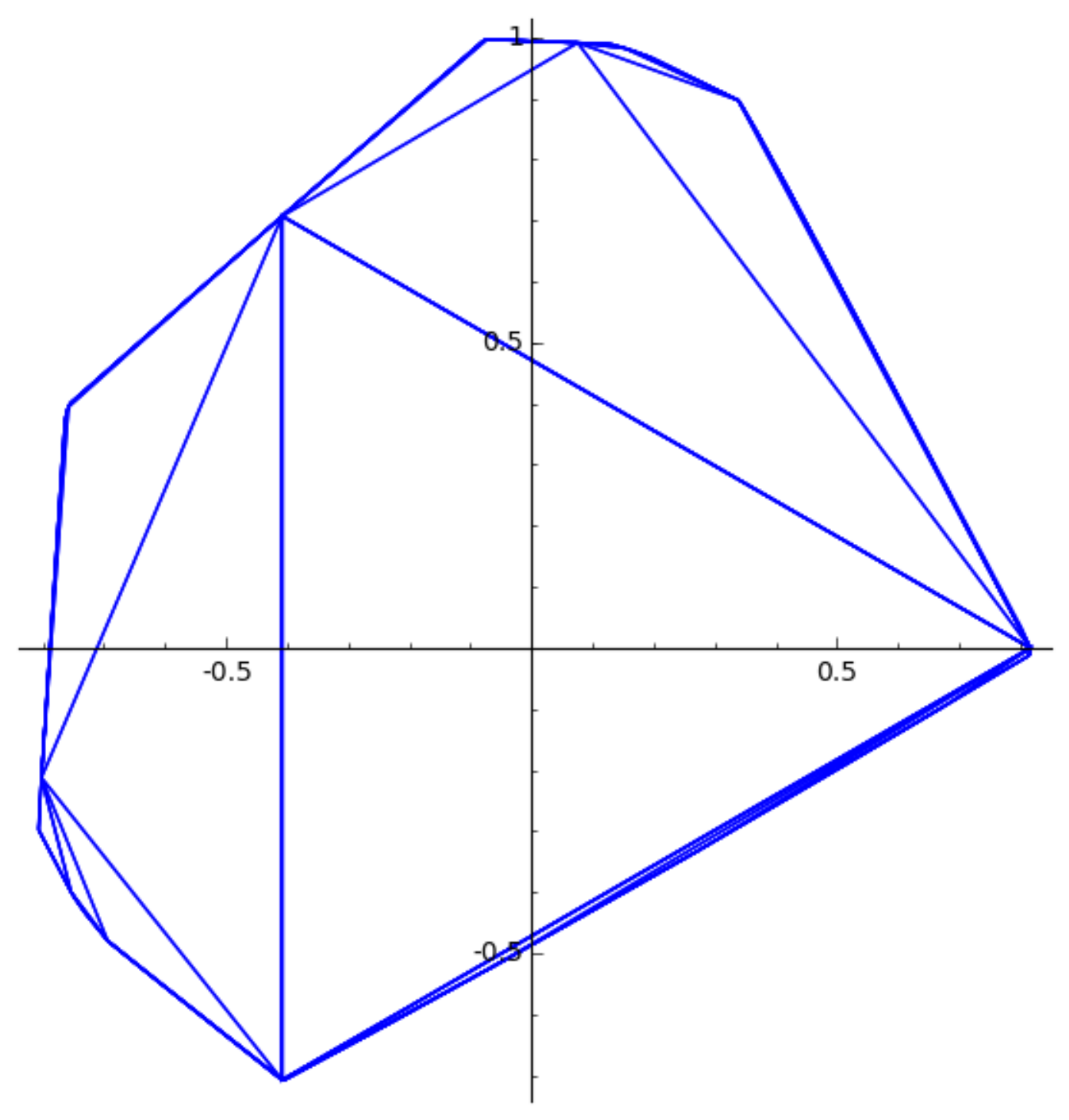} 
		\includegraphics[width=3.4cm]{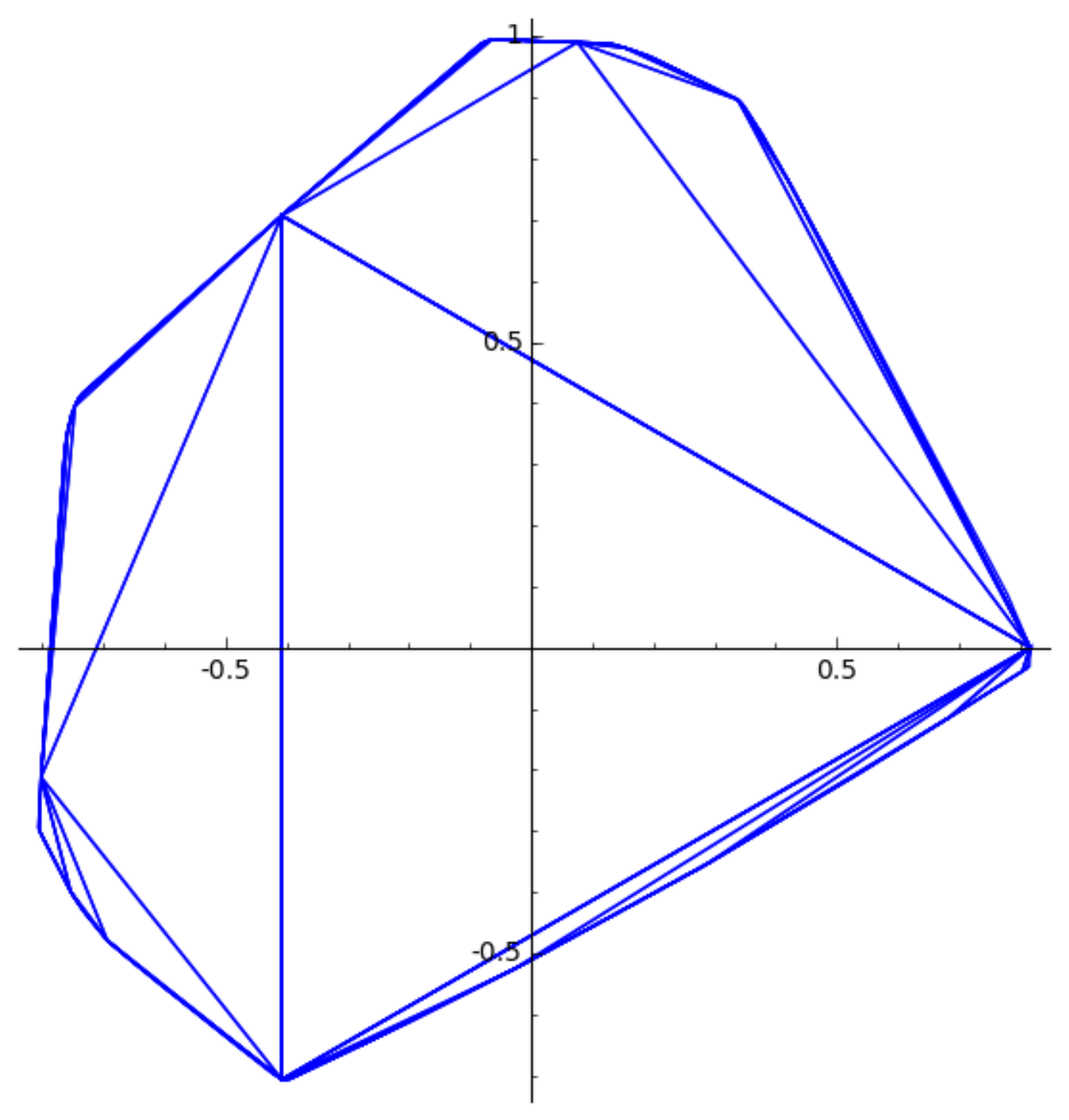} 
		\includegraphics[width=3.4cm]{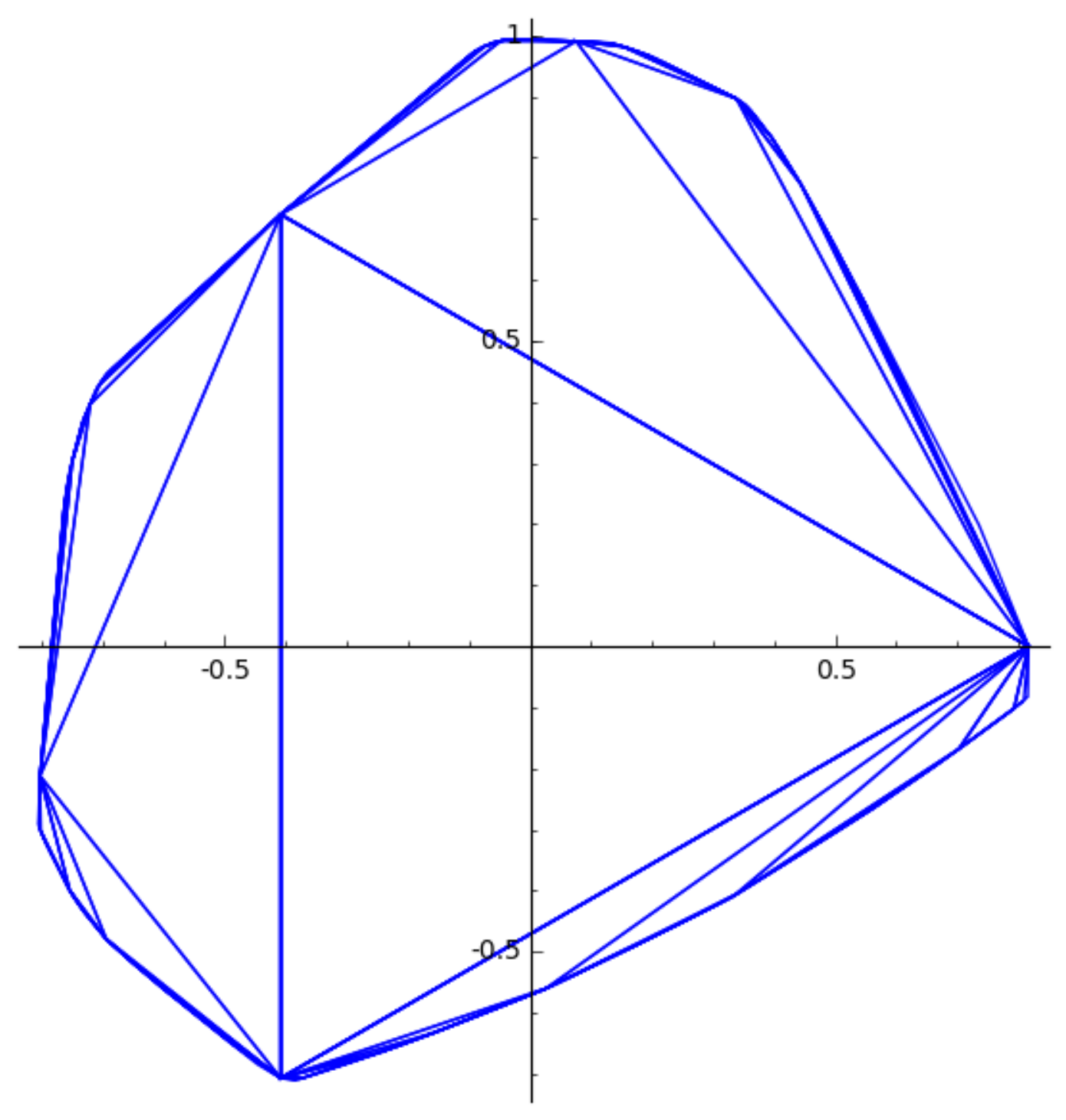} 
		\includegraphics[width=3.4cm]{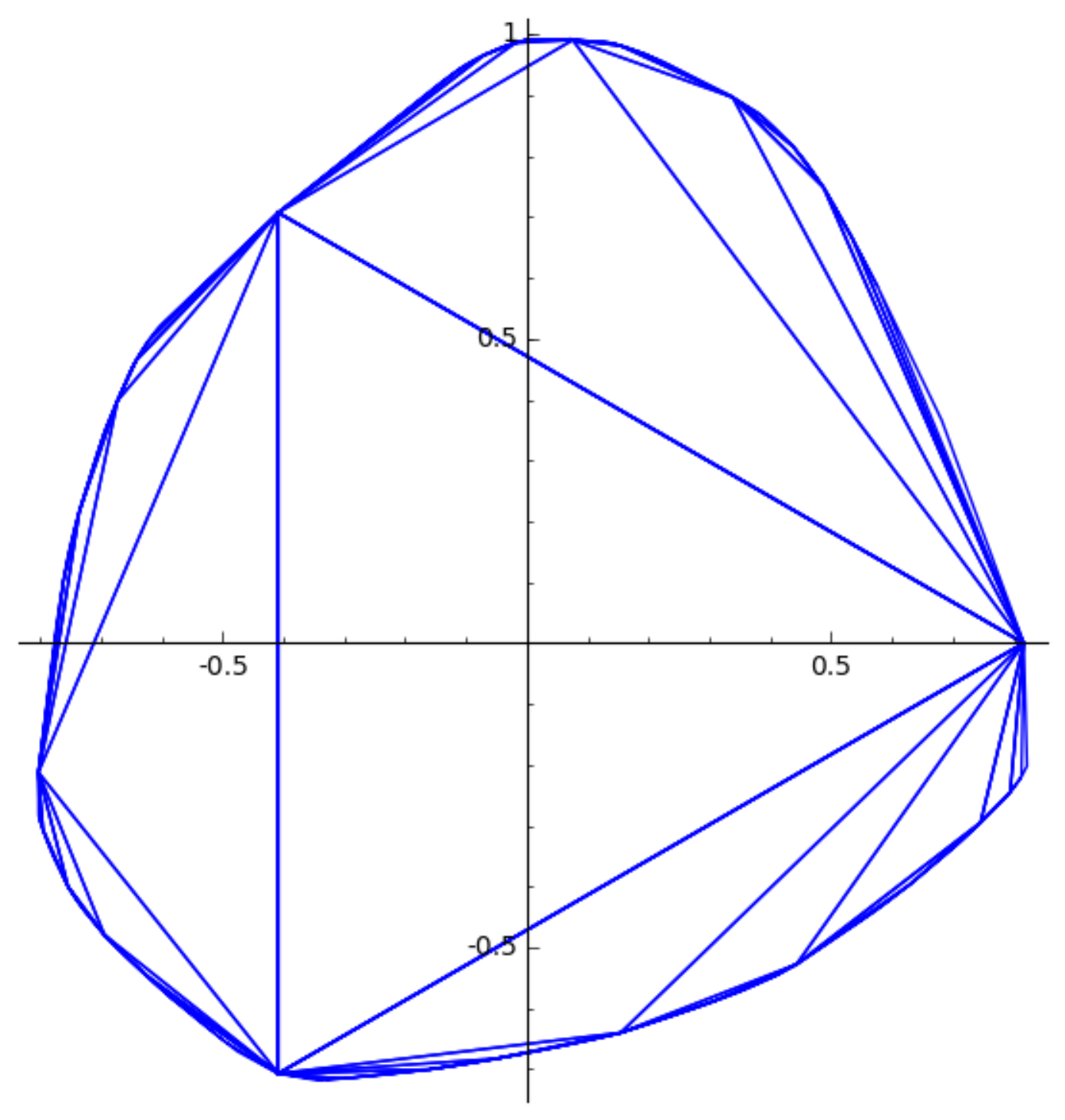}
 		\includegraphics[width=3.4cm]{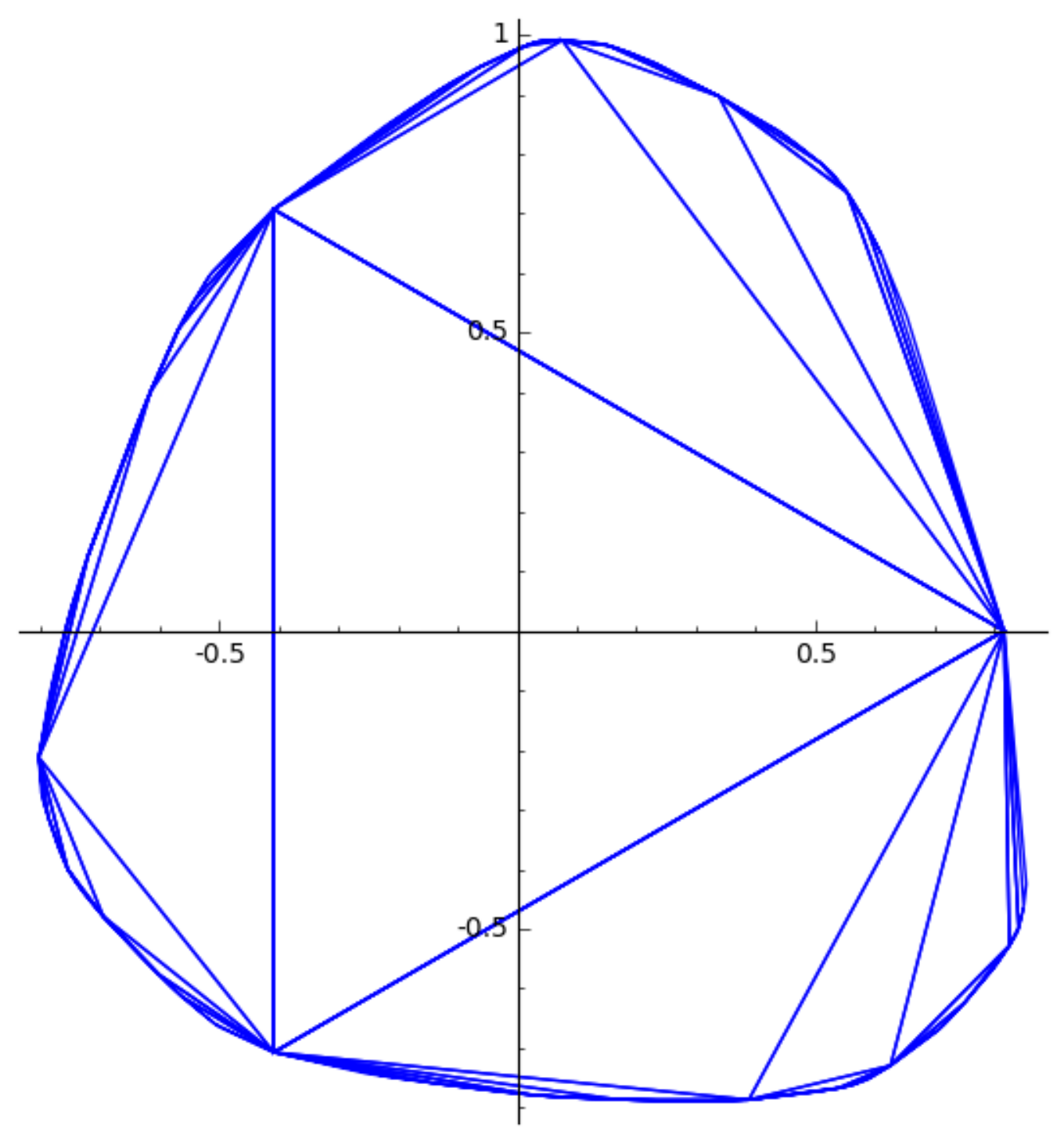} 
		\includegraphics[width=3.4cm]{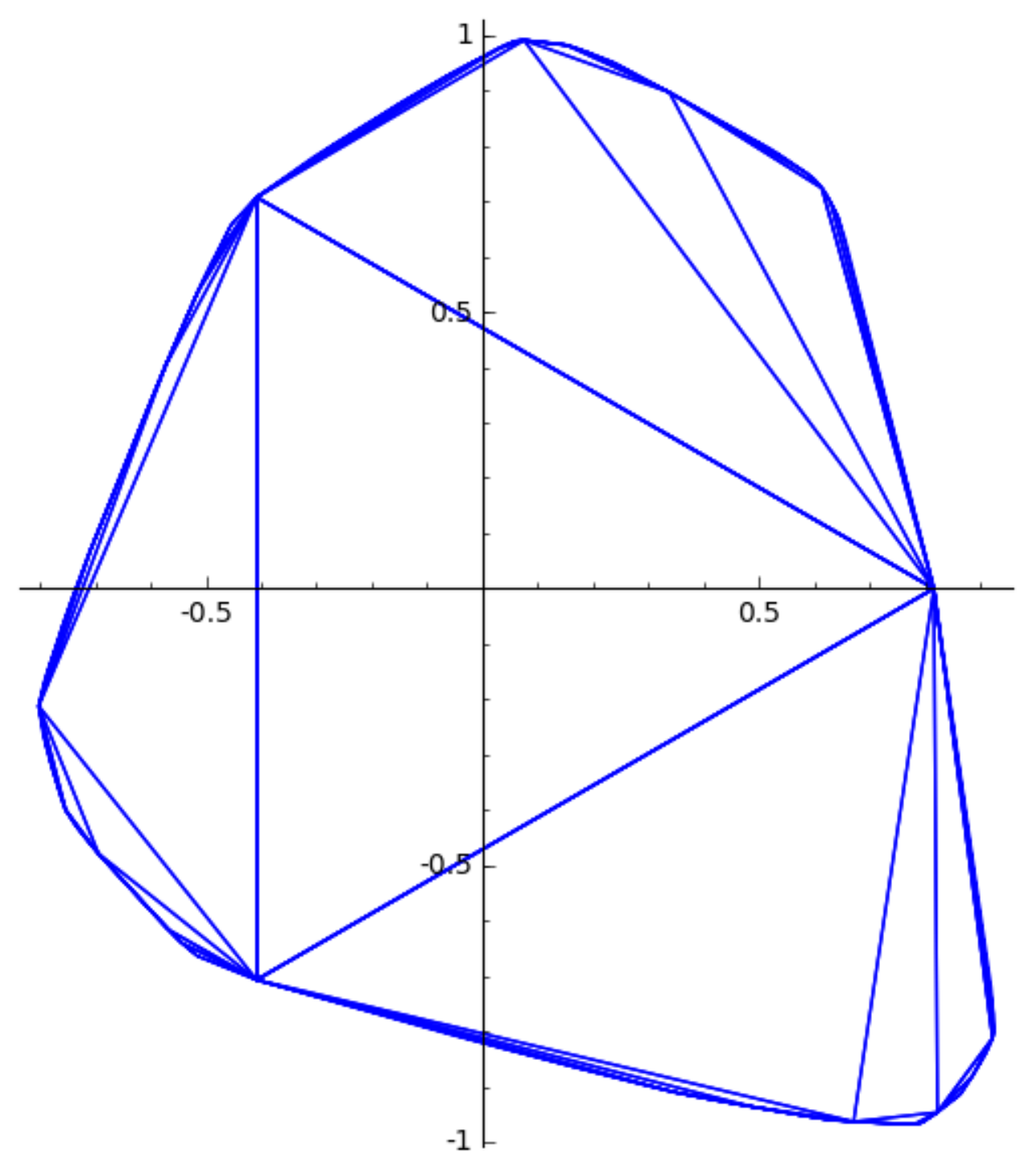} 
		\includegraphics[width=3.4cm]{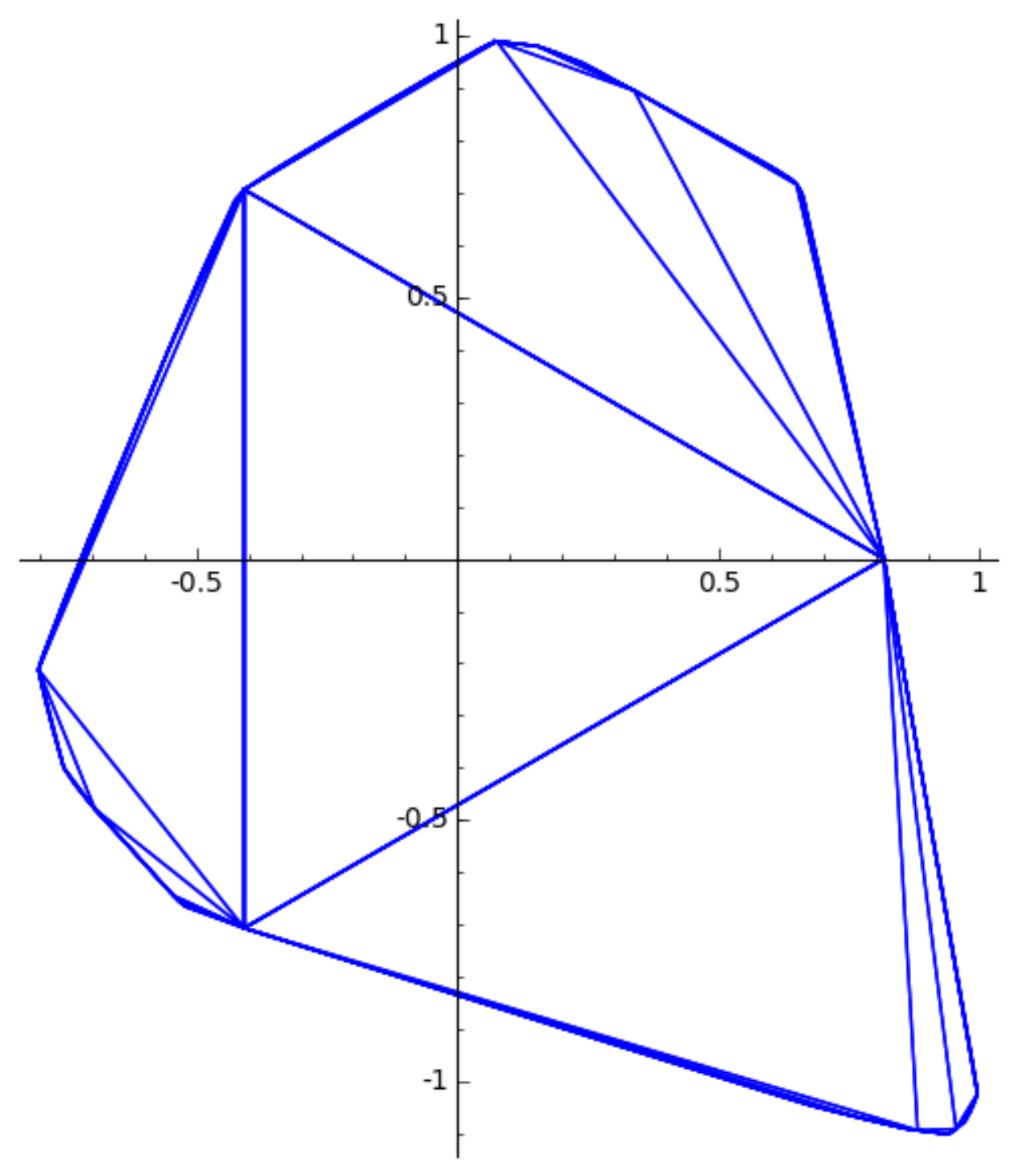} 
		\includegraphics[width=3.4cm]{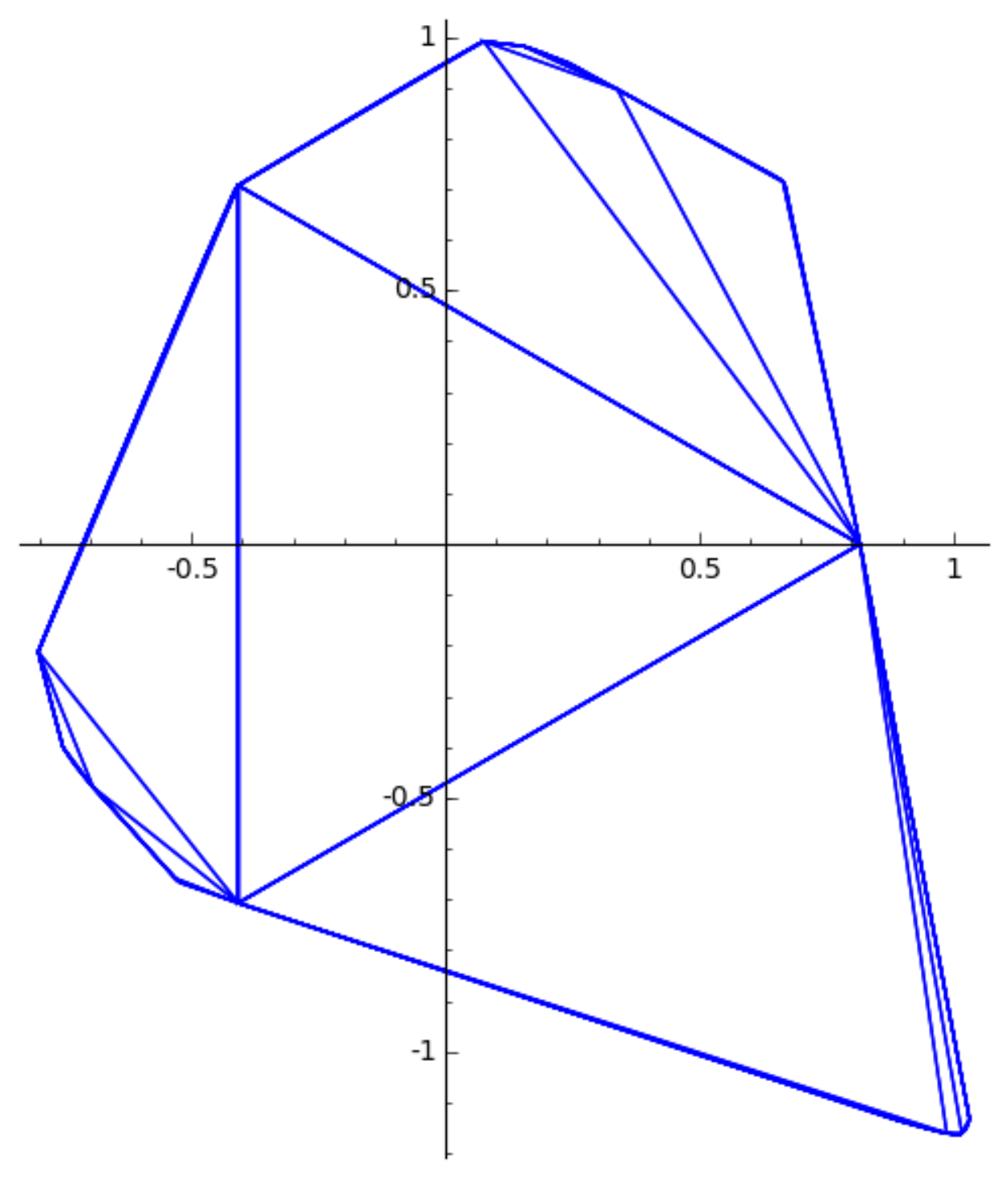}
		\caption{A degenerating sequence of projective structures in cloverleaf position all lying in the same cell of moduli space and with the first and last pictures close to the boundary at infinity. The parameters are
$(1/\sqrt[3]{2}, \sqrt[3]{2}, 1, \sqrt[3]{4}, 1, 1, 2^\mu, 2^{-\mu-2/3})$
for $\mu \in \{-2.5 + i/2 : 0 \leq i < 8\}$, and the domains appear to converge to polygons.
		}
	\label{fig:cloversequence}
\end{figure}

\begin{figure}[h]
	\centering	
 		\includegraphics[width=3.4cm]{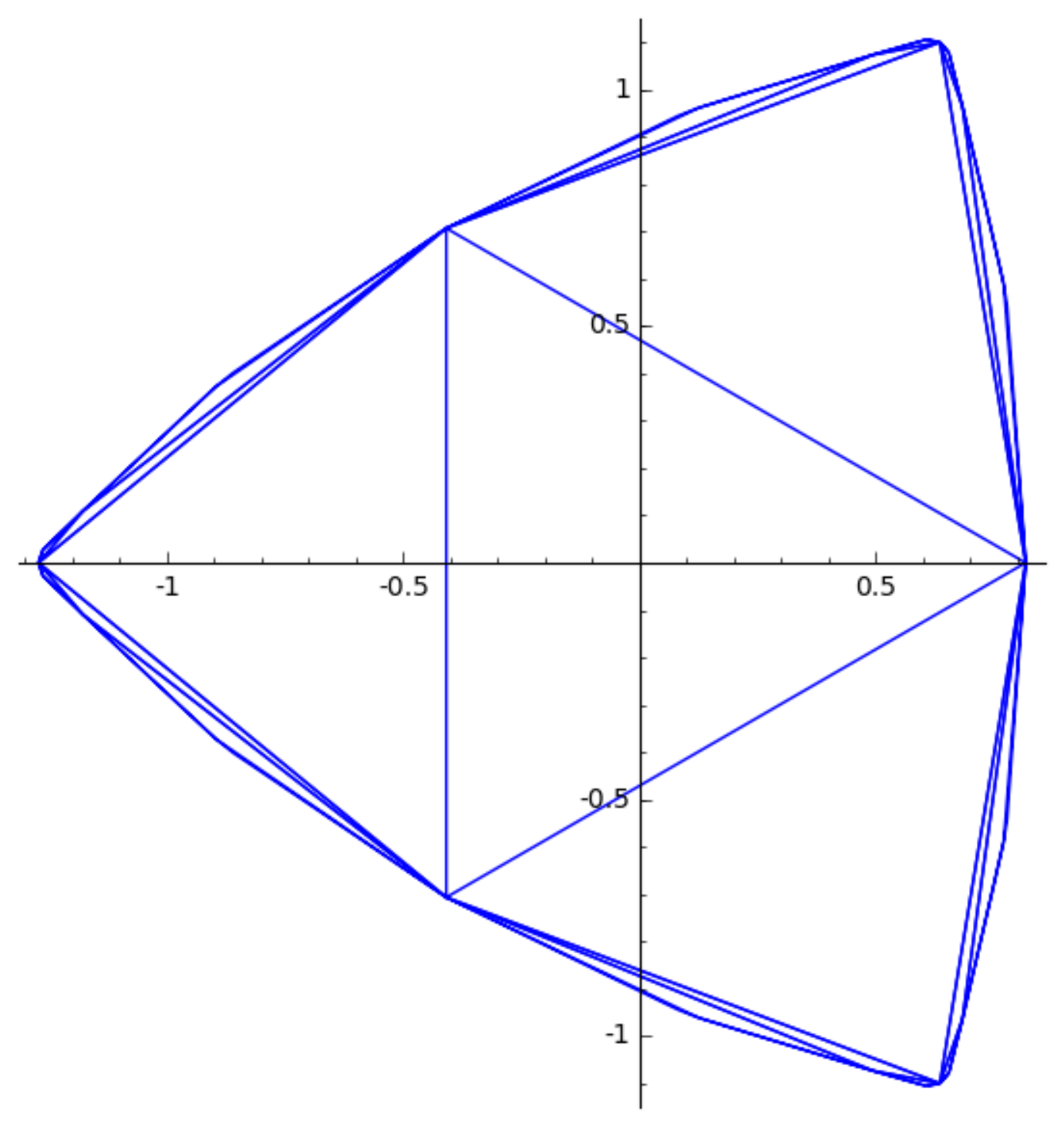} 
		\includegraphics[width=3.4cm]{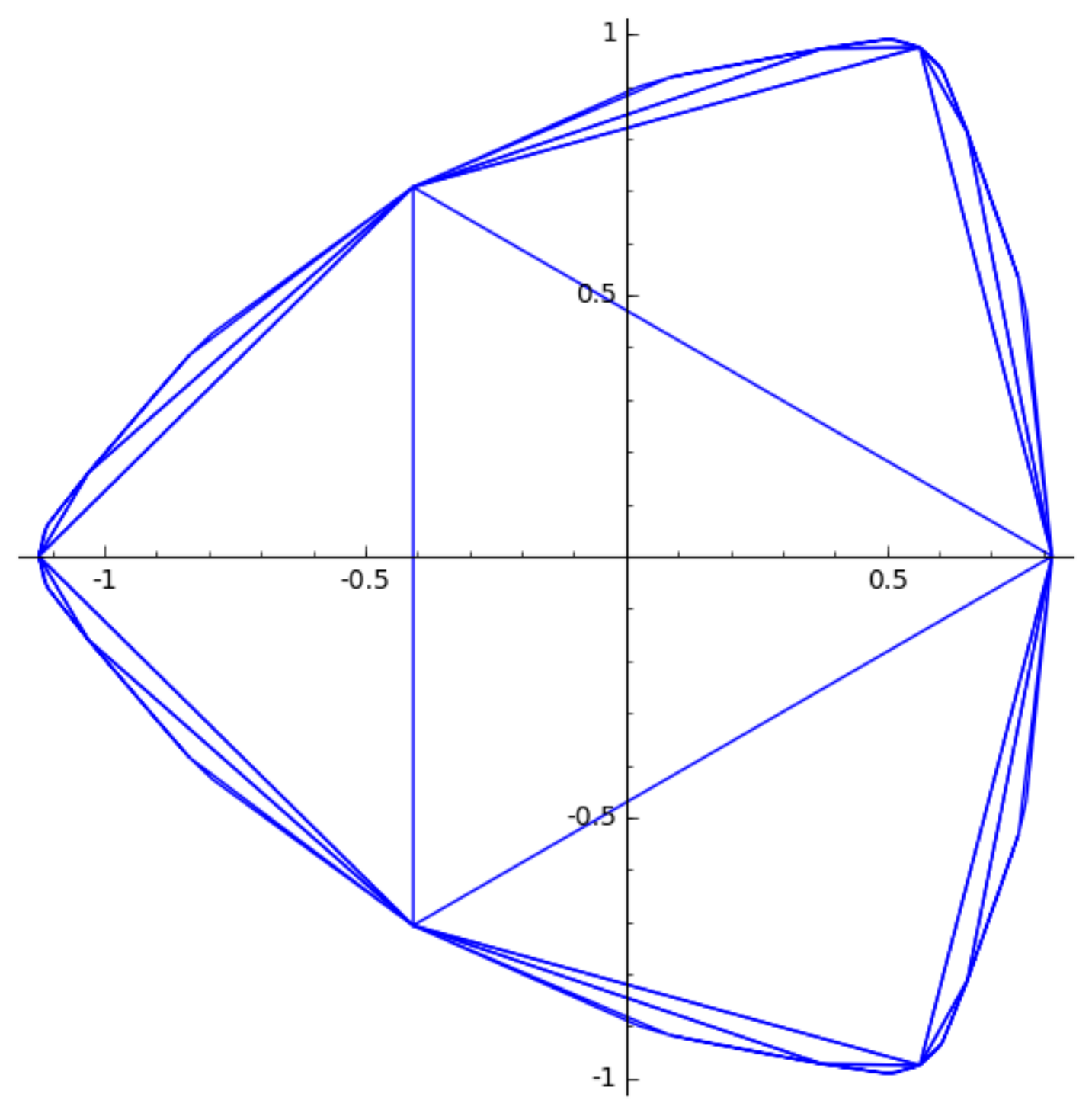} 
		\includegraphics[width=3.4cm]{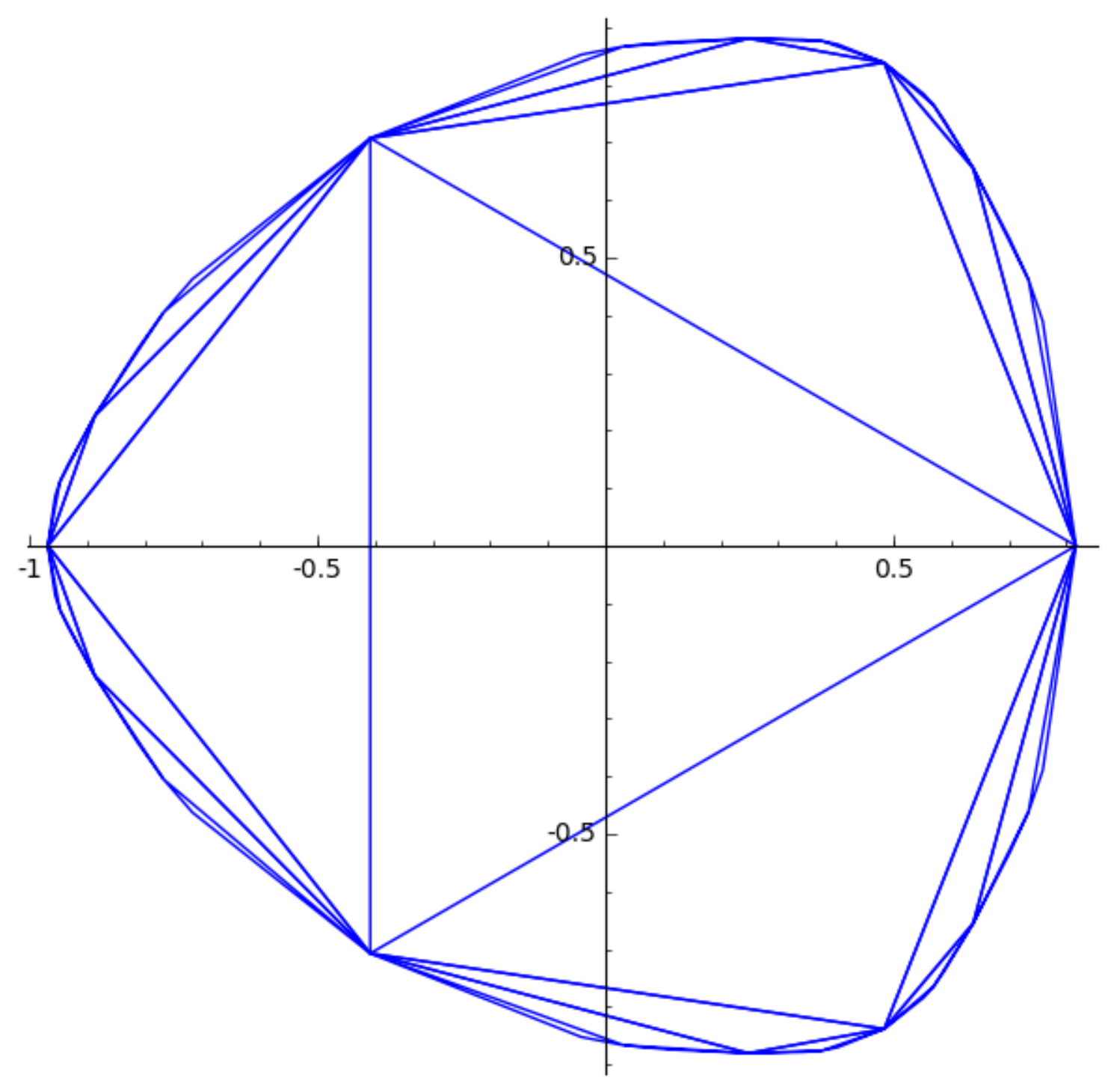} 
		\includegraphics[width=3.4cm]{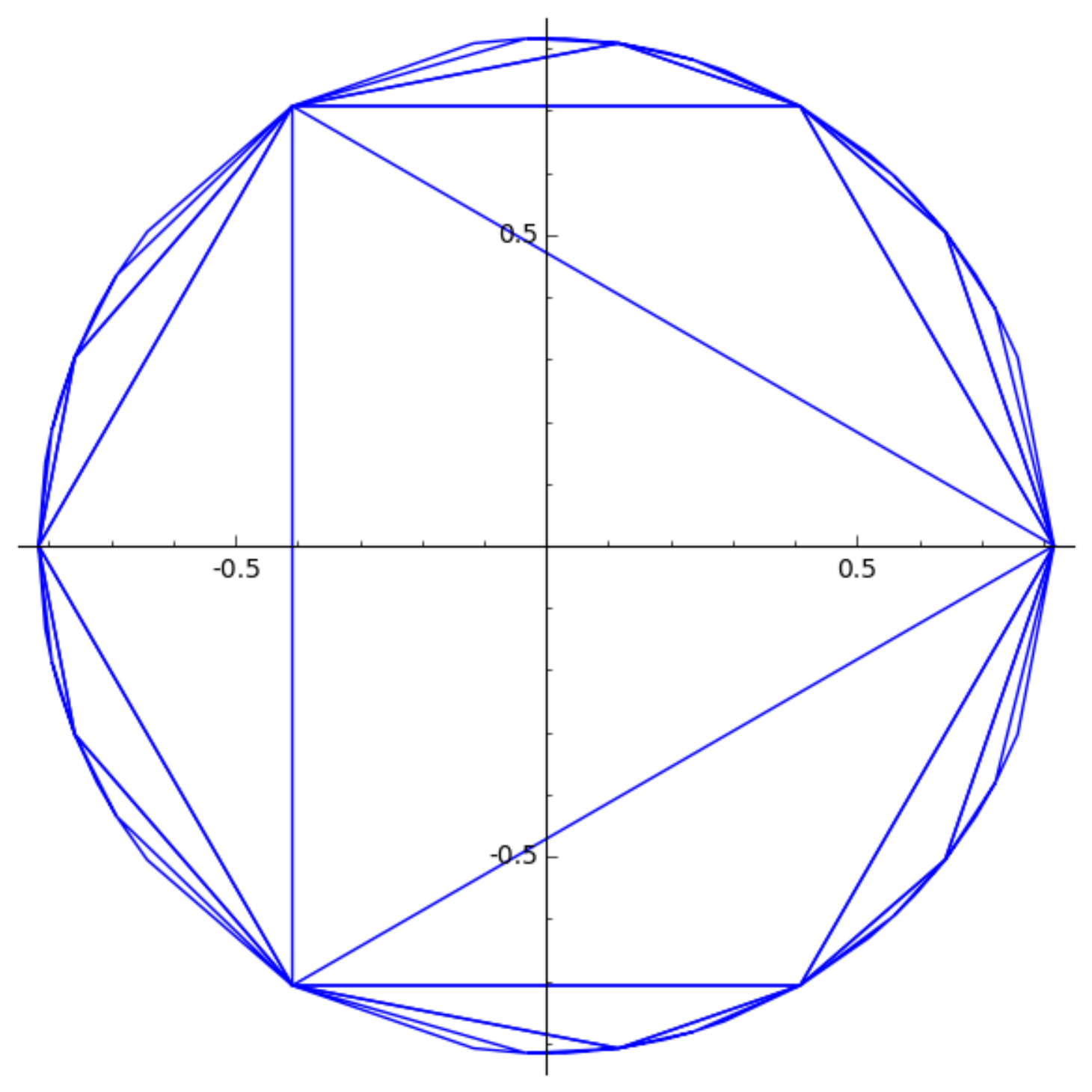}
 		\includegraphics[width=3.4cm]{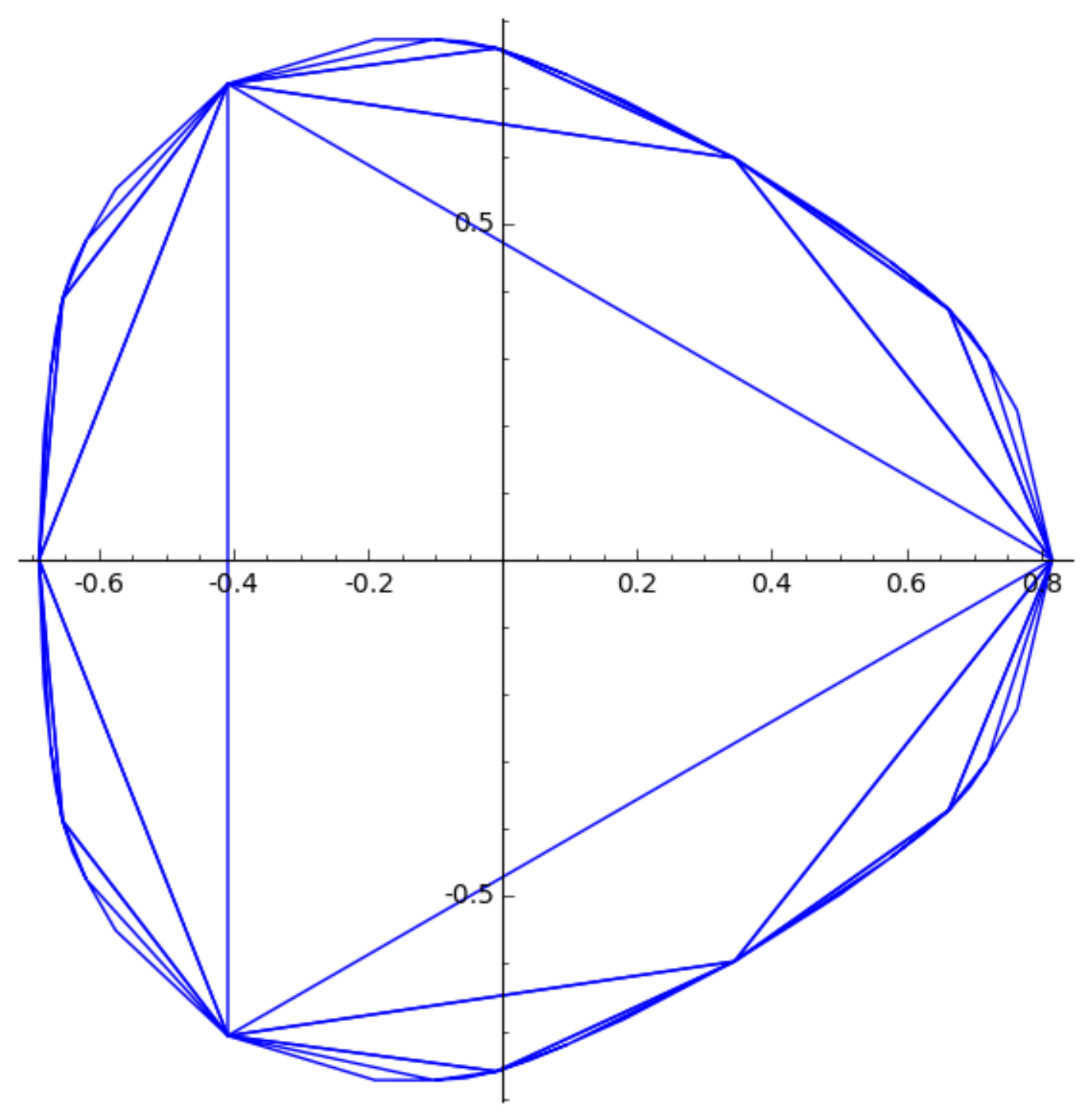} 
		\includegraphics[width=3.4cm]{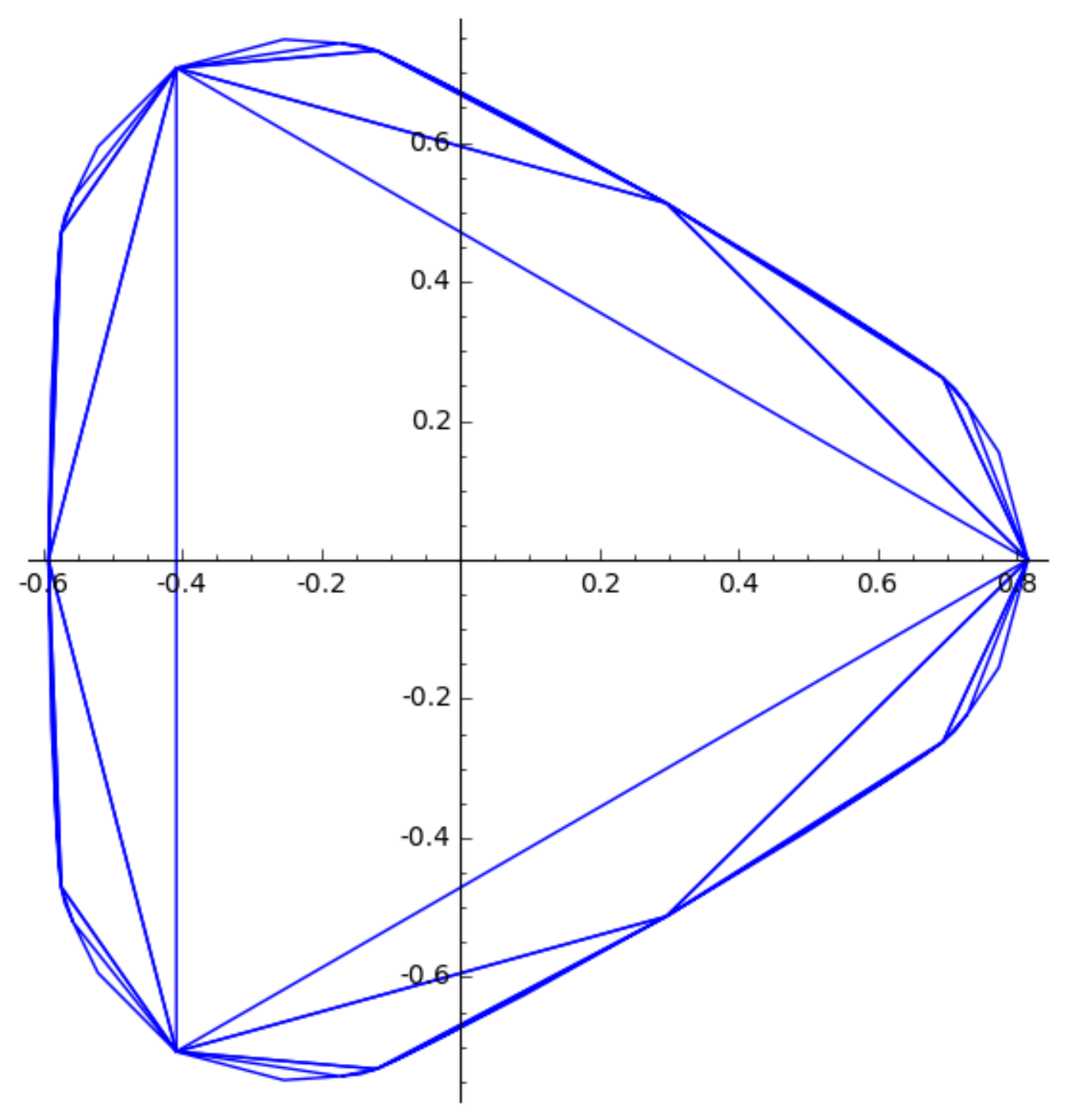} 
		\includegraphics[width=3.4cm]{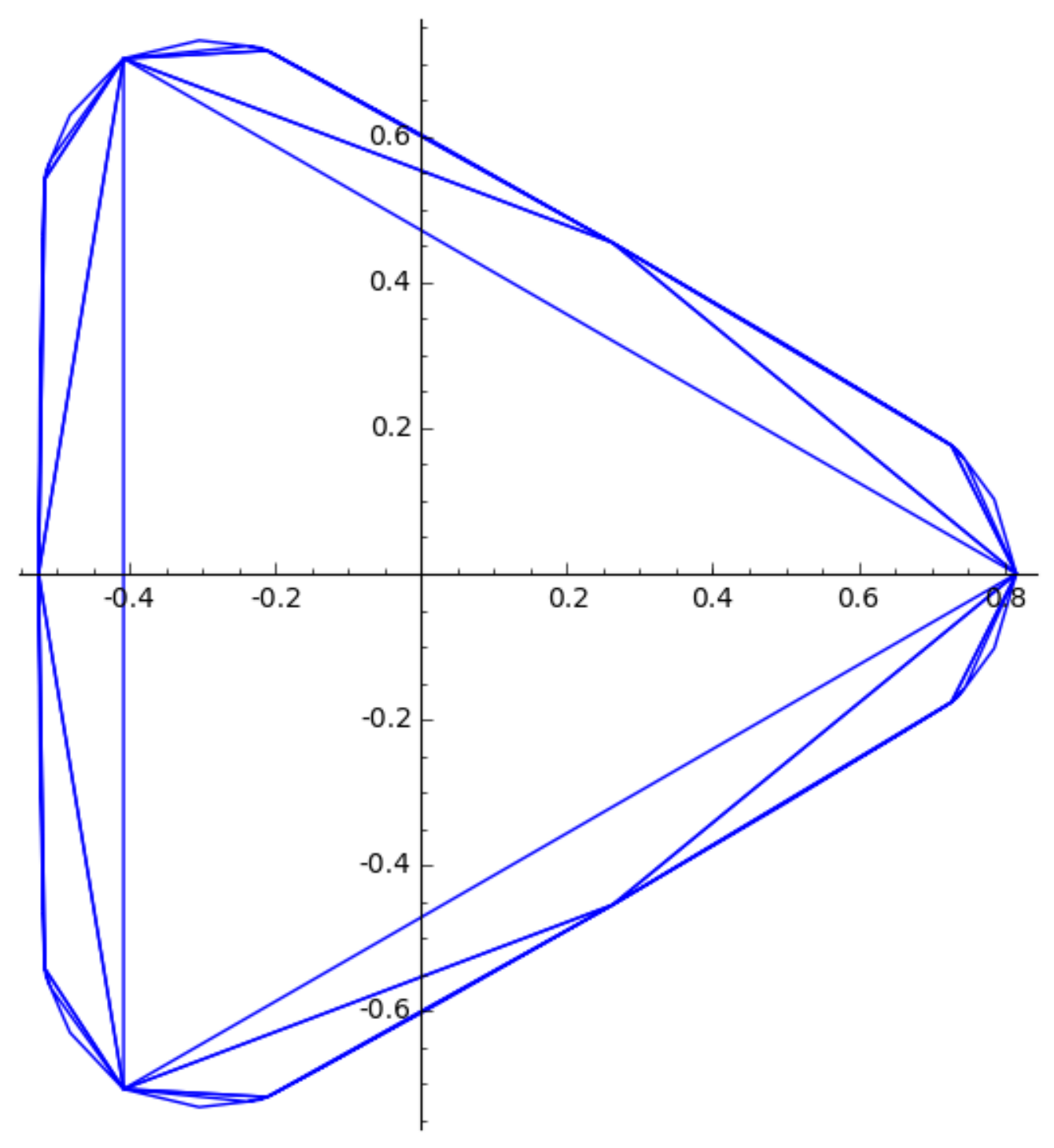} 
		\includegraphics[width=3.4cm]{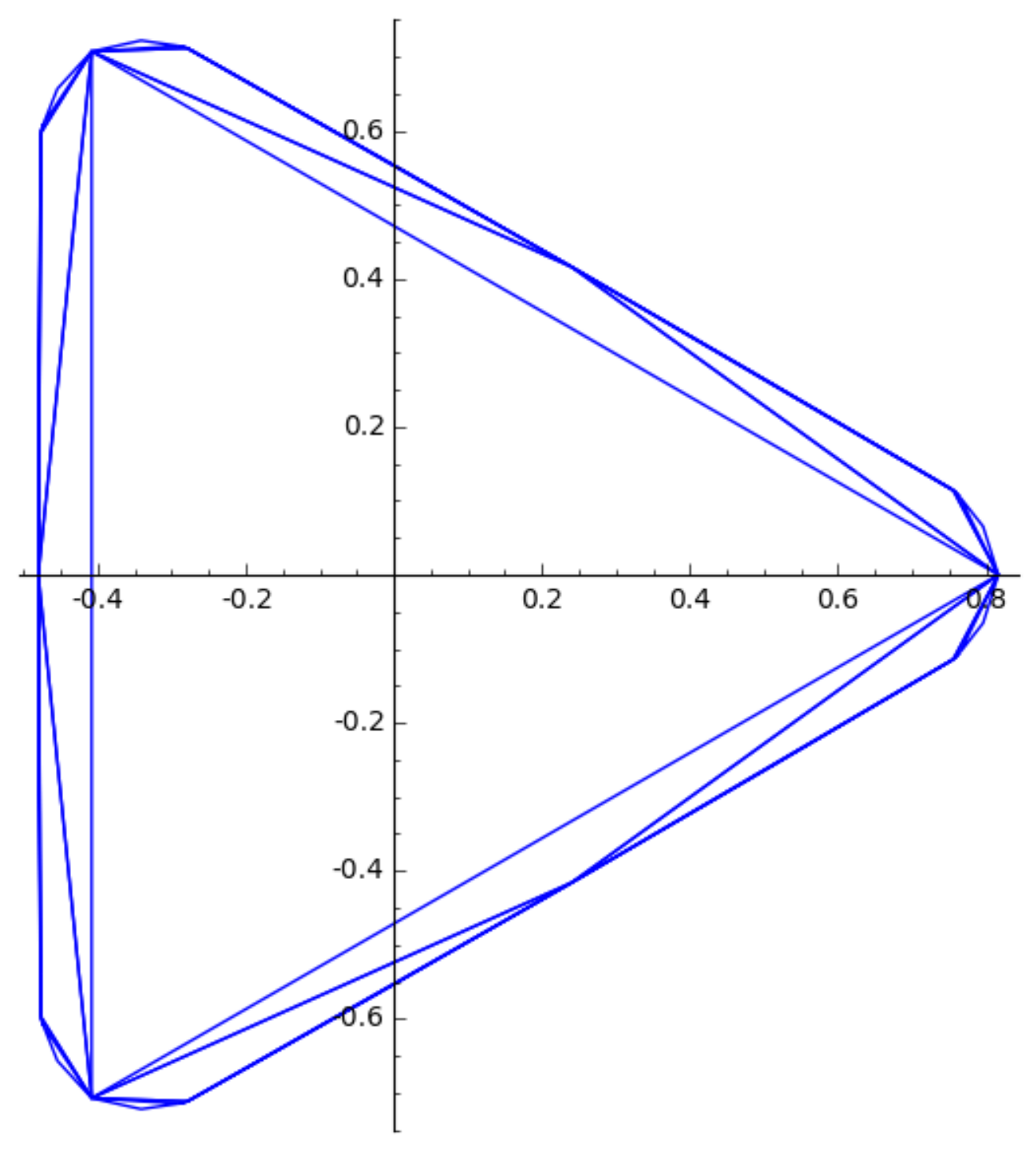}
		\caption{A degenerating sequence of projective structures on $S_{0,3}$ in clover position all lying in the same cell of moduli space and with the first and last pictures close to the boundary at infinity.}
	\label{fig:cloversequenceS03}
\end{figure}

This set-up, and the explicit determination of the cell decomposition for $S_{1,1}$, allow a systematic study of deformations and degenerations of strictly convex projective structures on the once-punctured torus, which will be conducted in the future. See Figure~\ref{fig:cloversequence} for an example. Moreover, there is scope to do explicit computations for other surfaces of low complexity---the main issue is that the rational functions appearing in the coordinates of orbits of the vertices of the fundamental domain become more complicated. Such examples are of particular interest because there are different possible approaches to compactifying the moduli space (just as there are different compactifications of Teichm\"uller space), whose relative merits can be explored with our tools. Moreover, there should be geometric invariants and properties associated with the relative position of a point in moduli space within the cell containing it. Again, it is hoped that our tools can be used to determine, and quantify, such invariants.


\section{Conclusion}
\label{sec:conclusion}

This paper gives evidence that results known about classical Teichm\"uller space may have analogues in projective geometry. In this paper, we have focussed on cusped strictly convex projective surfaces and highlighted possible applications in the previous section. More generally, the parameter space due to Fock and Goncharov also parameterises strictly convex projective structures with geodesic boundary. Cell decompositions for the analogous space of hyperbolic structures are known using a variety of  approaches; see, for instance, the work of Ushijima~\cite{Ushijima1999}, Penner~\cite{Penner2004}, Mondello~\cite{Mondello2009}, and Guo and Luo~\cite{GL2011}. Extending our methods to these more general surfaces would provide a unified framework, in which cusps can open to boundary components, and boundary components shrink to cusps.

An even more tantalising problem arises when going to higher dimensions. The constructions due to Epstein and Penner~\cite{MR918457} and Cooper and Long~\cite{CL} work in arbitrary dimensions. Moreover, there is a canonical cell decomposition of hyperbolic manifolds with boundary due to Kojima~\cite{koj1990, koj}. Whereas the moduli space of complete hyperbolic structures on a finite-volume hyperbolic manifold is a single point if the dimension is at least three, 
the moduli space of strictly convex projective structures on such a manifold may be larger as shown by Cooper, Long and Thistlethwaite~\cite{MR2264468, MR2372851, MR2529905}. At the time of writing, it is not clear why some hyperbolic 3--manifolds deform whilst others do not, and the study of the decomposition of the decorated moduli space may hold the key to the answer, as well as shed light on other connections between the geometry and topology of a manifold.


\subparagraph*{Acknowledgements}

This research was partially supported by Australian Research Council grant DP140100158.


\bibliography{cellrefs}


%

\address{School of Mathematics and Statistics F07\\ The University of Sydney\\ NSW 2006 Australia\\
robert.haraway@sydney.edu.au\\stephan.tillmann@sydney.edu.au}

\Addresses

\newpage

\appendix
\section{Code Listings}
\label{app:listings}
\begin{lstlisting}[caption={Maxima code, part 0},label=list:maxima0,captionpos=t,abovecaptionskip=-\medskipamount]
matrixEq(M,N) :=
  block([mr:length(M),mc:length(M[1]),
         nr:length(N),nc:length(N[1]),
         eqs:[],i,j],
        if mr = nr and mc = nc
        then for i:1 thru mr step 1 do
                 for j:1 thru mc step 1 do
                     eqs:cons(M[i][j]=N[i][j],eqs)
        else [],
        eqs)$
P : matrix([0,P01,P02],
           [P10,0,P12],
           [P20,P21,0])$
C : matrix([C0,0,0],[0,C1,0],[0,0,C2])$
D : matrix([D0,0,0],[0,D1,0],[0,0,D2])$
goal : matrix([0,f,1],[1,0,f],[f,1,0])$
pcdsols : solve(matrixEq(C.P.D,goal),[f,C0,C1,C2,D0,D1,D2])[3]$
print(pcdsols[1]);
\end{lstlisting}
(N.B. Maxima does not solve matrix
equations, necessitating the little
for loop routine.)

\begin{lstlisting}[caption={Maxima code, output 0},label=list:out0,captionpos=t,abovecaptionskip=-\medskipamount]
       1/3    1/3    1/3
    P01    P12    P20
f = --------------------
       1/3    1/3    1/3
    P02    P10    P21
\end{lstlisting}

\begin{lstlisting}[caption={Maxima code, part 1},label=list:maxima1,captionpos=t,abovecaptionskip=-\medskipamount]
symmsols : subst(P21^(1/3)*P12^(1/3)*P20^(-1/3)*P10^(-1/3),%r3,
                 pcdsols)$
symmsols[5];symmsols[6];symmsols[7];
\end{lstlisting}

\newpage

\begin{lstlisting}[caption={Maxima code, output 1},label=list:out1,captionpos=t,abovecaptionskip=-\medskipamount]
        1/3    1/3
     P12    P21
D0 = -------------
        1/3    1/3
     P10    P20

        1/3    1/3
     P02    P20
D1 = -------------
        1/3    1/3
     P01    P21

        1/3    1/3
     P01    P10
D2 = -------------
        1/3    1/3
     P02    P12
\end{lstlisting}

\begin{lstlisting}[caption={Maxima code, part 2},label=list:maxima2,captionpos=t,abovecaptionskip=-\medskipamount]
mapToStdTrigon(V,v) :=
  block([vV, L0, L1, L2, D, m],
    vV : v.V,
    L0 : (vV[2][3] * vV[3][2])^(1/3)/(vV[2][1] * vV[3][1])^(1/3),
    L1 : (vV[3][1] * vV[1][3])^(1/3)/(vV[3][2] * vV[1][2])^(1/3),
    L2 : (vV[1][2] * vV[2][1])^(1/3)/(vV[1][3] * vV[2][3])^(1/3),
    D : matrix([L0,0,0],
               [0,L1,0],
               [0,0,L2]),
    m : invert(V.D),
    determinant(m)^(-1/3)*m)$
projFromTo(V,v,W,w) := invert(mapToStdTrigon(W,w)) . 
                       mapToStdTrigon(V,v)$
\end{lstlisting}

We already know two flags of $y$: namely,
the first two standard basis vectors and
the first two associated covectors. This
is the content of listing \ref{list:maxima3}.

\begin{lstlisting}[caption={Maxima code, part 3},label=list:maxima3,captionpos=t,abovecaptionskip=-\medskipamount]
V0 : matrix([1],[0],[0])$
v0 : matrix([0,t012,1])$
V1 : matrix([0],[1],[0])$
v1 : matrix([1,0,t012])$
V2 : matrix([0],[0],[1])$
v2 : matrix([t012,1,0])$
\end{lstlisting}

The equations we solve are just our definitions
of the parameters as triple ratios, so we
had better code in triple ratios as well.
Listing \ref{list:maxima4} takes this
into consideration.

\begin{lstlisting}[caption={Maxima code, part 4},label=list:maxima4,captionpos=t,abovecaptionskip=-\medskipamount]
R3(v,V) :=
  block([vV:v.V],
    vV[1][2]*vV[2][3]*vV[3][1]/
   (vV[1][3]*vV[2][1]*vV[3][2]))$
trilateral(p,q,r) := addrow(addrow(p,q),r)$
triangle(P,Q,R) := addcol(addcol(P,Q),R)$

tr(P,p,Q,q,R,r) := R3(trilateral(p,q,r),triangle(P,Q,R))$
\end{lstlisting}

\begin{lstlisting}[caption={Maxima code, part 5},label=list:maxima5,captionpos=t,abovecaptionskip=-\medskipamount]
maprhs(listofeqs) := create_list(rhs(eq),eq,listofeqs)$
\end{lstlisting}

In a functional programming language like Haskell
we could write \texttt{(map rhs)} and skip listing \ref{list:maxima5}.

\begin{lstlisting}[caption={Maxima code, part 6},label=list:maxima6,captionpos=t,abovecaptionskip=-\medskipamount]
U2repEqs : block([UU,UUV0,UUV1,V2V0,V2V1,e01eq,e10eq,fullsol,Usol],
  UU : matrix([UU0],[UU1],[UU2]),
  UUV0 : matrix([UUV00,UUV01,UUV02]),
  UUV1 : matrix([UUV10,UUV11,UUV12]),
  V2V0 : matrix([0,1,0]),
  V2V1 : matrix([1,0,0]),
  e01eq : e01^3 = tr(V0,v0,UU,UUV1,V2,V2V1),
  e10eq : e10^3 = tr(V1,v1,V2,V2V0,UU,UUV0),
  subst((e10^3+1)*t012^2,%r8,solve([e01eq,e10eq,UUV0.UU=0,UUV0.V0=0,
                                      UUV1.UU=0,UUV1.V1=0],
                         [UU0,UU1,UU2,
                          UUV00,UUV01,UUV02,
                          UUV10,UUV11,UUV12])[2]));
\end{lstlisting}
N.B. we are allowed to substitute what we like for \texttt{\%r8}
because of our scalar freedom in choosing a representative of $U_2$, and we
choose Maxima's second solution since the other solutions are spurious.

The next listing incorporates the values of $U_0,$ $U_1$ and $U_2$.

\begin{lstlisting}[caption={Maxima code, part 7},label=list:maxima7,captionpos=t,abovecaptionskip=-\medskipamount]
U2rep : matrix([(e10^3 + 1)*t012^2],
               [(e01^3 + 1)*e10^3],
               [-e10^3*t012]);
sigma3 : matrix([0,0,1],
               [1,0,0],
               [0,1,0]);
U0rep : sigma3.sublis([e10=e21,e01=e12],U2rep);
U1rep : sigma3.sublis([e21=e02,e12=e20],U0rep);
\end{lstlisting}

We have now found the other vertices of our configurations $c,m,y$,
so we have all the triangle parts. Listing \ref{list:maxima8}
expresses this.

\begin{lstlisting}[caption={Maxima code, part 8},label=list:maxima8,captionpos=t,abovecaptionskip=-\medskipamount]
Y : triangle(U2rep,V1,V0)$
C : triangle(V1,U0rep,V2)$
M : triangle(V0,V2,U1rep)$
\end{lstlisting}

\begin{lstlisting}[caption={Maxima code, part 9},label=list:maxima9,captionpos=t,abovecaptionskip=-\medskipamount]
uu : matrix([uu0,uu1,uu2])$
yy : trilateral(uu,v1,v0)$
factor(subst(t012^2*t210^3,%r35,
             solve([R3(yy,Y)=t210^3,uu.U2rep=0],
                   [uu0,uu1,uu2]))[1]);
\end{lstlisting}

The next listing incorporates the computation of the covectors.

\begin{lstlisting}[caption={Maxima code, part 10},label=list:maxima10,captionpos=t,abovecaptionskip=-\medskipamount]
u2rep : matrix([e01^3*e10^3, t012^2*t210^3, 
                t012*(e01^3*t210^3 + t210^3 + e01^3*e10^3 + e01^3)]);
u0rep : sublis([e01=e12,e10=e21],u2rep).invert(sigma3);
u1rep : sublis([e12=e20,e21=e02],u0rep).invert(sigma3);
\end{lstlisting}

\begin{lstlisting}[caption={Maxima code, part 11},label=list:maxima11,captionpos=t,abovecaptionskip=-\medskipamount]
y : trilateral(u2rep,v1,v0)$
c : trilateral(v1,u0rep,v2)$
m : trilateral(v0,v2,u1rep)$

r : factor(projFromTo(M,m,Y,y))$
g : factor(projFromTo(Y,y,C,c))$
b : factor(projFromTo(C,c,M,m))$
\end{lstlisting}

\begin{lstlisting}[caption={Maxima code, part 12},label=list:maxima12,captionpos=t,abovecaptionskip=-\medskipamount]
perph : r.g.b$
factor(charpoly(perph,lambda));
\end{lstlisting}

Listing \ref{list:maxima13} 
phrases the contents of Lemma~\ref{lem:parhol} in code, assuming all variables are positive.

\begin{lstlisting}[caption={Maxima code, part 13},label=list:maxima13,captionpos=t,abovecaptionskip=-\medskipamount]
parhol : [t210 = 1/t012, e02 = 1/(e01*e10*e12*e21*e20)]$
\end{lstlisting}

The next listing summarises the choice of light-cone representatives in \S\ref{sec:canonicity}.

\begin{lstlisting}[caption={Maxima code, part 14},label=list:maxima14,captionpos=t,abovecaptionskip=-\medskipamount]
S0 : matrix([e20*e02*e21],[0],[0])$
S1 : matrix([0],[e01*e10*e02],[0])$
S2 : matrix([0],[0],[e12*e21*e10])$
omega : matrix([e01*e10*e12, e12*e21*e20, e20*e02*e01])$
\end{lstlisting}

The next listing encodes the bendings, using the semantics of subtractive primary colours.

\begin{lstlisting}[caption={Maxima code, part 15},label=list:maxima15,captionpos=t,abovecaptionskip=-\medskipamount]
YB : omega.r.S0$
YB : factor(YB);
CB : omega.g.S1$
CB : factor(CB);
MB : omega.b.S2$
MB : factor(MB);
\end{lstlisting}

\begin{lstlisting}[caption={Maxima code, part 16},label=list:maxima16,captionpos=t,abovecaptionskip=-\medskipamount]
factor(solve(sublis(parhol,CB)=1, [e20]));
\end{lstlisting}

\begin{lstlisting}[caption={Maxima code, part 17},label=list:maxima17,captionpos=t,abovecaptionskip=-\medskipamount]
factor(solve(sublis(parhol,YB)=1,[e20]));
\end{lstlisting}

Algorithms for cylindrical algebraic
decomposition can return a list
containing a point from every cell 
of this decomposition. We run
such an algorithm in {\tt Sage}\rm~\cite{sage} on the
intersection of the cyan and yellow
flat-set projections as in listing \ref{list:SAGE},
returning as output listing \ref{list:SAGEout}.
The first line is just from the declaration of variables.
The last line is the output of \texttt{qepcad}; it is
a list of a point from every cell in the intersection
of the projections to the $t_{0 1 2},e_{0 1},e_{1 0},e_{1 2},e_{2 1}$-plane
of the cyan and yellow flat-sets. But this list is empty; therefore, their intersection is empty.

\begin{lstlisting}[caption={SAGE code},label=list:SAGE,captionpos=t,abovecaptionskip=-\medskipamount]
# SAGE code for disjointness
qf = qepcad_formula
var('t012,e01,e10,e12,e21')

top_c = e01*e10^2*e12^2*e21*t012 - e12^3 - 1
bot_c = e21^3*t012 + t012 - e01^2*e10*e12*e21^2
cyan_pos = qf.and_(top_c > 0, bot_c > 0)
cyan_neg = qf.and_(top_c < 0, bot_c < 0)
cyan_0   = qf.and_(top_c== 0, bot_c== 0)

top_y = e01*e10^3*e12^2*e21*t012 + e01*e12^2*e21*t012 - e10
bot_y = e01*t012 - e01^3*e10*e12*e21^2 - e10*e12*e21^2
yell_pos = qf.and_(top_y > 0, bot_y > 0)
yell_neg = qf.and_(top_y < 0, bot_y < 0)
yell_0   = qf.and_(top_y== 0, bot_y== 0)

all_pos = qf.and_(t012>0, e01>0, e10>0, e12>0, e21>0)

cyan   = qf.and_(all_pos,qf.or_(cyan_pos,cyan_0,cyan_neg))
yellow = qf.and_(all_pos,qf.or_(yell_pos,yell_0,yell_neg))

qepcad(qf.and_(cyan,yellow),solution='cell-points')
\end{lstlisting}

\begin{lstlisting}[caption={SAGE output},label=list:SAGEout,captionpos=t,abovecaptionskip=-\medskipamount]
(t012, e01, e10, e12, e21)
[]
\end{lstlisting}


\end{document}